\documentclass[a4paper,UKenglish,cleveref, autoref, thm-restate]{lipics-v2021}

\title{Minimal Arrangements of Spherical Geodesics} 

\author{Giovanni Viglietta}{University of Aizu, Japan}{viglietta@gmail.com}{https://orcid.org/0000-0001-6145-4602}{}

\authorrunning{G. Viglietta} 



\category{} 

\relatedversion{} 

\nolinenumbers 

\newtheorem{open}{Open Problem}

\hideLIPIcs

\begin{document}

\maketitle

\begin{abstract}
We study arrangements of geodesic arcs on a sphere, where all arcs are internally disjoint and each arc has its endpoints located within the interior of other arcs. We establish fundamental results concerning the minimum number of arcs in such arrangements, depending on local geometric constraints such as \emph{one-sidedness} and \emph{$k$-orientation}.

En route to these results, we generalize and settle an open problem from CCCG~2022, proving that any such arrangement has at least two \emph{clockwise swirls} and at least two \emph{counterclockwise swirls}. 
\end{abstract}

\section{Introduction}\label{sec:1}
\subsection{Art Gallery Problem and Spherical Diagrams}\label{sec:1.1}
In Discrete and Computational Geometry, the \emph{Art Gallery Problem} asks how many ``light sources'' are required to ``illuminate'' a given geometric environment~\cite{orourke,urrutia2}. The environment may be an enclosure, such as a polygon or a polyhedron, or more broadly, a spatial arrangement of ``opaque'' objects that obstruct light rays.

In recent years, a line of research on the Art Gallery Problem has focused on illuminating $3$-dimensional polyhedra by choosing a subset of their edges as light sources. In this setting, the facets of a polyhedron are interpreted as opaque obstacles, and one wants to illuminate the polyhedron's interior with as few ``edge lights'' as possible~\cite{urrutia,orthoguard,cano,2reflex}.

To this end, a technique developed in~\cite{cano} involves selecting a small set of edges whose endpoints collectively include all vertices of the polyhedron. Interestingly, however, these edges may not be sufficient to fully illuminate the polyhedron's interior, as there is an abundance of polyhedra having internal points that do not directly see any vertices~\cite{orourke,spherical1,spherical2}. We say that such points are \emph{vertex-hidden}. Thus, the study of vertex-hidden points and the combinatorics of their visibility maps becomes central in~\cite{cano}, as well as in the development of a general theory of visibility-related problems for polyhedral objects in $3$-space~\cite{mini}.

In order to systematize these fundamental investigations, Spherical Occlusion Diagrams (SODs) were introduced as a model for visibility maps of vertex-hidden viewpoints relative to \emph{polygonal scenes}, i.e., arrangements of interior-disjoint polygons in $3$-space~\cite{spherical1,spherical2}.

We give some preliminary definitions. A \emph{geodesic arc} on a sphere is the unique shortest curve connecting two non-antipodal points. Clearly, any geodesic arc lies within a great circle on the sphere. We say that a geodesic arc $a$ \emph{blocks} a geodesic arc $b$ if an endpoint of $b$ lies in the relative interior of $a$. If $a$ blocks $b$, we say that $b$ \emph{hits} $a$.

\begin{definition}\label{d:1a}
A \emph{Spherical Diagram (SD)} is a finite non-empty collection $\mathcal D$ of pairwise interior-disjoint geodesic arcs on the unit sphere in $\mathbb R^3$, such that each arc of $\mathcal D$ is blocked by arcs of $\mathcal D$ at each endpoint.
\end{definition}

\begin{definition}\label{d:1b}
A \emph{Spherical Occlusion Diagram (SOD)} is a Spherical Diagram $\mathcal D$ with the additional property that all arcs in $\mathcal D$ that hit the same arc of $\mathcal D$ reach it from the same side.
\end{definition}

\begin{figure*}
\centering
\includegraphics[scale=0.55]{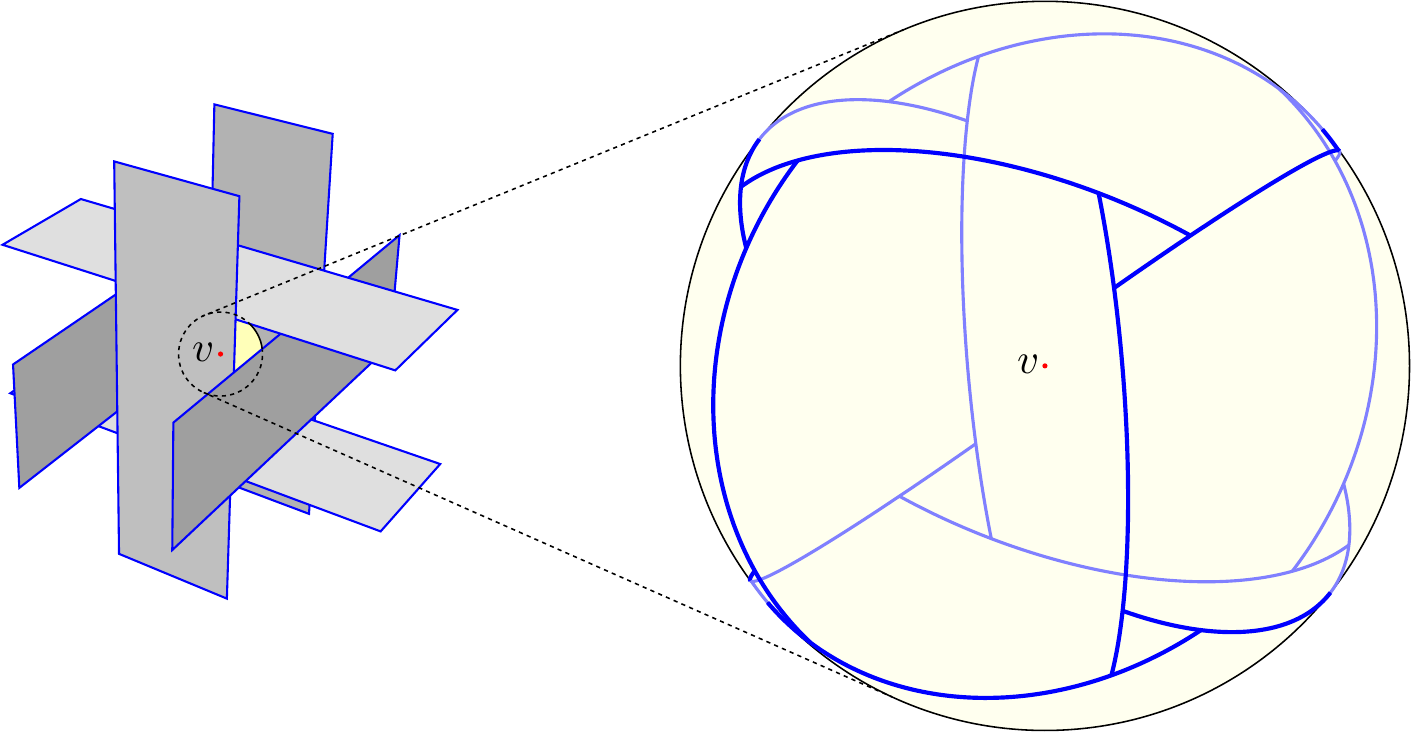}
\caption{A polygonal scene of six rectangles (left), where the central point $v$ is vertex-hidden. The portion of each edge that is visible to $v$ is radially projected onto a sphere centered at $v$, creating a geodesic arc. The resulting visibility map is an SOD (right).}
\label{fig:intro}
\end{figure*}

SDs may also be referred to as \emph{two-sided} arrangements of arcs, whereas SODs are \emph{one-sided}, to stress the fact that their arcs are hit only from one side (the term ``one-sided'' was previously used in~\cite{eppstein} in the context of rectangular layouts). While SODs were introduced and studied in~\cite{mini,spherical1,spherical2}, in this paper we also set out to investigate SDs as a natural generalization. Notably, most of what is currently known about SODs is more generally true for SDs (refer to \cref{sec:previous}), and SDs are interesting geometric objects in their own right.

As illustrated in \cref{fig:intro}, the visibility map of a vertex-hidden viewpoint in a polygonal scene is necessarily an SOD (see~\cite{spherical2}). Interestingly, the converse is not always the case, as was shown in~\cite{kimberly}. Nonetheless, SODs were instrumental in proving that any vertex-hidden viewpoint in a polygonal scene must see at least eight distinct edges (this number reduces to six for viewpoints that are not vertex-hidden). This result was obtained in~\cite{mini} by showing that any SOD consists of at least eight arcs. In turn, this follows from the observation that any SOD (in fact, any SD) must have at least four \emph{swirls}, which are defined next.

\begin{definition}\label{d:swirl}
A \emph{swirl} in an SD is a cycle of arcs, each of which hits the next (and such that the last hits the first), going either all clockwise or all counterclockwise. The \emph{degree} of a swirl is the number of arcs constituting it.
\end{definition}
As an example, the SOD in \cref{fig:intro} has four clockwise swirls and four counterclockwise swirls, all of degree three.

\subsection{Polyhedra with Restricted Edge Orientations}\label{sec:1.2}
In the study of the Art Gallery Problem, it is customary to investigate not only polygons or polyhedra in their full generality, but also specific subclasses with particular geometric properties. For example, alongside general polyhedra, also \emph{orthogonal polyhedra} have been considered (\cref{fig:poly.a}): these are polyhedra whose faces meet at right angles~\cite{urrutia,orthoguard,2reflex}.

Similarly, a \emph{$k$-face-oriented} polyhedron has the property that there exist $k$ unit vectors such that the normal of each face is parallel to one of these vectors~\cite{face}. Note that orthogonal polyhedra are $3$-face-oriented, and any polyhedron is $k$-face-oriented for a large-enough $k$.

In the same vein, we may define a polygonal scene to be \emph{$k$-edge-oriented} if there exist $k$ unit vectors such that every edge of every polygon in the scene is parallel to one of the vectors. Again, any orthogonal polyhedron is a $3$-edge-oriented polygonal scene; in general, any $k$-face-oriented polyhedron is also a ${k\choose 2}$-edge-oriented polygonal scene (\cref{fig:poly.c}).

\begin{figure*}
\centering
\begin{subfigure}{0.3\textwidth}
    \includegraphics[scale=0.75]{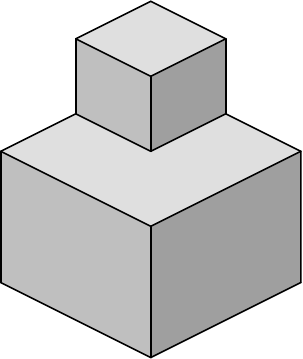}
    \subcaption{}
    \label{fig:poly.a}
\end{subfigure}
\hfill
\begin{subfigure}{0.3\textwidth}
    \includegraphics[scale=0.75]{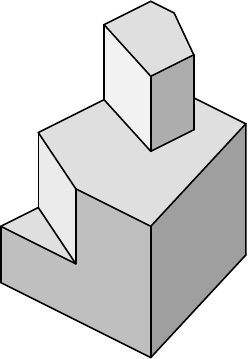}
    \subcaption{}
    \label{fig:poly.b}
\end{subfigure}
\hfill
\begin{subfigure}{0.3\textwidth}
    \includegraphics[scale=0.75]{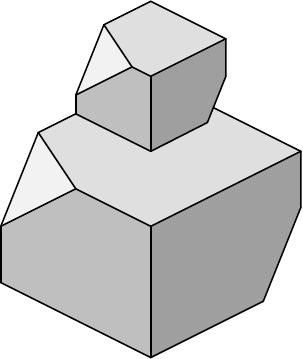}
    \subcaption{}
    \label{fig:poly.c}
\end{subfigure}
\caption{(a) An orthogonal polyhedron. (b) A $7$-edge-oriented polyhedron having only vertical and horizontal edges (and faces). (c) A $4$-face-oriented and $6$-edge-oriented polyhedron.}
\label{fig:poly}
\end{figure*}

Following the seminal result of~\cite{mini} mentioned in \cref{sec:1.1}, our motivating question is: {\bf What is the minimum number of edges that are necessarily visible to a vertex-hidden viewpoint in a $k$-edge-oriented polygonal scene?} In order to model the visibility maps of such viewpoints, we give the following definition.

\begin{definition}\label{d:orient}
An SD $\mathcal D$ is \emph{$k$-oriented} if there exists a set $P$ of $k$ points on the unit sphere, called \emph{poles}, no two of which are antipodal, as well as a function $f\colon \mathcal D\to P$ such that each arc $a\in \mathcal D$ is collinear with the pole $f(a)$, but contains neither $f(a)$ nor its antipodal point.
\end{definition}
With the notation of \cref{d:orient}, we refer to the point antipodal to a pole $p\in P$ as the \emph{anti-pole} of $p$. Also, for any arc $a\in \mathcal D$, the pole $f(a)$ and its anti-pole $-f(a)$ are said to be the \emph{vanishing points} of $a$.\footnote{The term ``vanishing point'' comes from perspective theory, and indicates a point where parallel lines in $3$-space appear to converge when projected onto a surface such as a plane or, in this case, a sphere.} If three of the $k$ poles lie on a same great circle, then $\mathcal D$ is said to be \emph{degenerate}. Otherwise, $\mathcal D$ is \emph{non-degenerate}.

It is immediate to recognize that the visibility map of any vertex-hidden point of a $k$-edge-oriented polygonal scene is a $k$-oriented SOD. The converse is not true in general (a counterexample is found in~\cite{kimberly}), but it may be true for sufficiently small values of $k$.

Thus, we recast our motivating question: {\bf What is the minimum number of arcs that a $k$-oriented SD or SOD can have?} While this formulation may not be strictly equivalent to the one about polygonal scenes, it stands as a question of independent interest in extremal graph theory~\cite{bollobas}, because SDs can be characterized as spherical drawings of certain graphs. Consider a non-empty $3$-regular planar graph $G$ that can be drawn on a sphere by means of internally disjoint geodesic arcs in such a way that each vertex of $G$ is incident to two collinear arcs, and any chain of such collinear arcs is itself a geodesic arc. Such a drawing coincides with an SD in which no two arcs have a common endpoint, and vice versa. For the purpose of minimizing arcs, this is not a restrictive assumption on SDs; also, it is easy to see that minimizing arcs in SDs is equivalent to minimizing vertices (or edges) in such graphs. $k$-oriented SDs and SODs can be characterized in similar ways, as well.

\subsection{Statement of Results}\label{sec:1.3}
In this paper, we give a complete answer to the above question for all \emph{non-degenerate} SDs and SODs. That is, {\bf for every $k$, we provide non-degenerate $k$-oriented SDs and SODs having the minimum possible number of arcs}.

We are also able to do so for \emph{degenerate} $k$-oriented SDs and SODs of all possible configurations, provided that $k\leq 5$. For $k\geq 6$ we have sporadic results, but classifying and analyzing the numerous degenerate configurations of the poles remains a challenge.

In addition, we also determine the minimum number of \emph{swirls} that a $k$-oriented SD or SOD may have in all of the aforementioned cases (with one exception; see below).

Our results are summarized in \cref{tab:1}.

\begin{table}[h!]
\centering
\begin{tabular}{ |c|c||c|c|c| }
\hline
$k$-orientation & alignment & swirls & SD arcs & SOD arcs \\ 
\hline
\hline
$3$-oriented & $(2,2,2)$ & $8$ & $12$ & $12$ \\ 
\hline
$4$-oriented & $(2,2,2,3)$ & $6$ & $9$ & $11$ \\
$4$-oriented & $(3,3,3,3)$ & $6$ & $9$ & $10$ \\
\hline
$5$-oriented & $(2,2,2,2,4)$ & $6$ & $9$ & $11$ \\
$5$-oriented & $(2,3,3,3,3)$ & $5$ & $8$ & $9$ \\
$5$-oriented & $(3,3,3,4,4)$ & $4$--$5$ & $8$ & $9$ \\
$5$-oriented & $(4,4,4,4,4)$ & $4$ & $8$ & $9$ \\
\hline
$(k \geq 6)$-oriented & $(2,2,\dots,2,k-1)$ & $6$ & $9$ & $11$ \\
$6$-oriented & $(3,3,3,3,3,3)$ & $4$ & $6$ & $8$ \\
$6$-oriented & $(3,3,3,4,4,4)$ & $4$ & $6$ & $8$ \\
$6$-oriented & $(3,4,4,4,4,5)$ & $4$ & $6$ & $8$ \\
$6$-oriented & $(4,4,4,5,5,5)$ & $4$ & $6$ & $8$ \\
$(k \geq 6)$-oriented & $(\geq 5,\geq 5,\dots, \geq 5)$ & $4$ & $6$ & $8$ \\
\hline
\end{tabular}

\vspace{0.4cm}
\caption{Minimum numbers of swirls and arcs in $k$-oriented SDs (two-sided arrangements) and SODs (one-sided arrangements), depending on their alignment. The list includes all non-degenerate configurations for every $k$, as well as all degenerate configurations up to $k=5$.}\label{tab:1}
\end{table}

To uniformly describe degenerate and non-degenerate $k$-oriented SDs with respect to the configuration of their poles, we use the concept of \emph{alignment}. A $k$-oriented SD with poles $p_1$, $p_2$, \dots, $p_k$ has alignment $(d_1, d_2,\dots, d_k)$ if, for every $1\leq i\leq k$, there are exactly $d_i$ distinct great circles that contain $p_i$ and another pole (we may assume $d_1\leq d_2\leq\dots\leq d_k$).

Therefore, non-degenerate $k$-oriented SDs have alignment $(k-1,k-1,\dots,k-1)$. Another important alignment is $(2,2,\dots,2,k-1)$, which corresponds to the visibility map of a polyhedron having only vertical and horizontal edges (\cref{fig:poly.b}), as well as $(3,3,3,3,3,3)$, which corresponds to the visibility map of a $4$-face-oriented polyhedron (\cref{fig:poly.c}).

Notably, for nearly all of the alignment configurations listed in \cref{tab:1}, there are SDs and SODs that simultaneously minimize the number of arcs and the number of swirls. The only exceptions are $(3,3,3,4,4)$, where we do not know if four swirls are possible, and $(4,4,4,4,4)$, where any SD or SOD minimizing the number of arcs has five swirls (the minimum is four).

Although the last row of \cref{tab:1} was partially established in~\cite{mini,spherical1,spherical2}, a question was left open: Does any SOD have at least two clockwise and two counterclockwise swirls? In this paper we answer it in the affirmative, not only for SODs, but more generally for all SDs.

The technical proofs of most preliminary results are found in \cref{asec:walks,asec:doubling,asec:swirl,asec:attractors}. Concluding remarks and several open problems are in \cref{asec:open}.

\section{Basic Constructions and Preliminary Results}\label{sec:3}
\subsection{Previous Results}\label{sec:previous}
We summarize all previous results on SDs (see also \cref{asec:previous}). Proofs of the following statements are  found in~\cite{mini,spherical2}; although they were stated and proved only for SODs, the reader may verify that none of the proofs makes use of the one-sidedness of the arrangements.

Every arc in an SD hits exactly two distinct arcs (one at each endpoint), and no two arcs in an SD intersect in more than one point. The union of all the arcs in an SD is a connected set. An SD with $n$ arcs partitions the unit sphere into $n+2$ spherically convex regions called \emph{tiles}, and no tile contains antipodal points. The union of the arcs of a swirl separates the unit sphere in two regions, exactly one of which is spherically convex; this region is called the \emph{eye} of the swirl. Given an SD, the interior of any great semicircle on the unit sphere is crossed by at least one arc of the SD, and the interior of any hemisphere contains the eye of at least one swirl. The \emph{swirl graph} of an SD is the undirected multigraph on the set of swirls having an edge between two swirls for every arc shared by them; this graph is planar. Every SD has at least four swirls, including a clockwise one and a counterclockwise one.

In addition to the above properties, SODs also enjoy the following ones. The swirl graph of any SOD is simple, planar, and bipartite. Also, every SOD has at least eight arcs.

\begin{figure*}
\centering
\begin{subfigure}{0.49\textwidth}
    \includegraphics[scale=0.5]{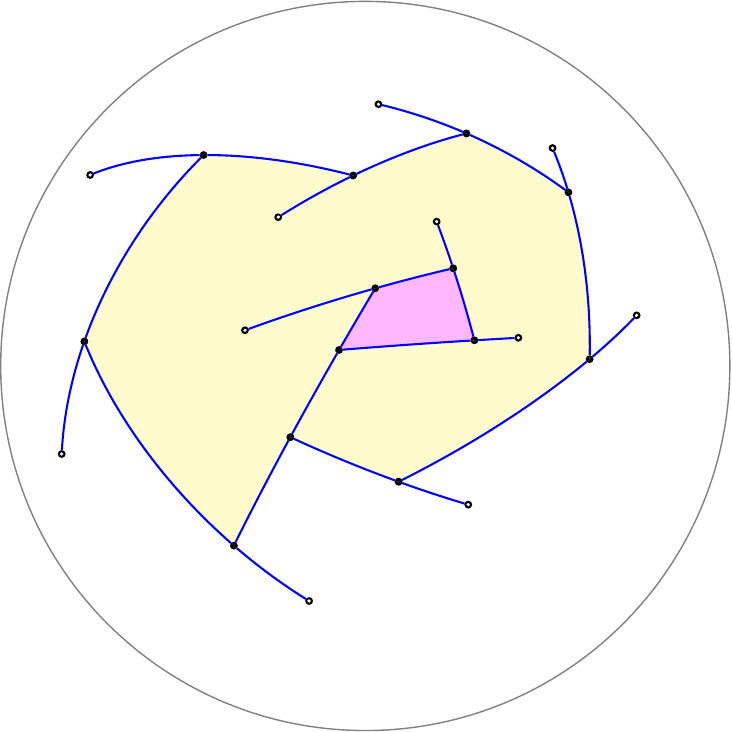}
    \subcaption{}
    \label{fig:1a}
\end{subfigure}
\hfill
\begin{subfigure}{0.49\textwidth}
    \includegraphics[scale=0.5]{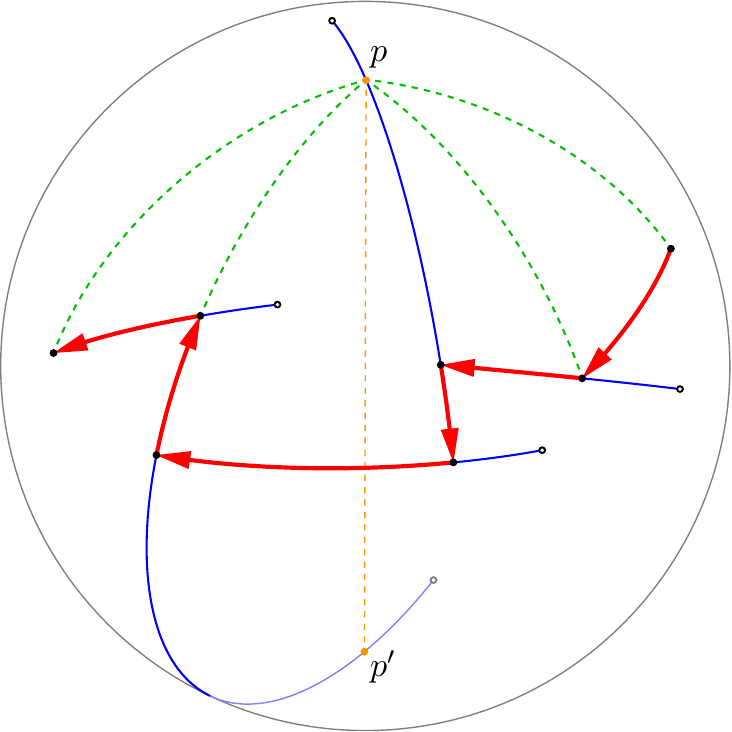}
    \subcaption{}
    \label{fig:1b}
\end{subfigure}
\caption{(a) The area in yellow is the right-side region of a sliding walk. The eye of a clockwise swirl (in purple) can be found within this region by doing a right-handed sliding walk from its boundary. (b) The initial steps of a clockwise walk with fulcrum $p$ (in red).}
\label{fig:1}
\end{figure*}

\subsection{Sliding Walks}\label{sec:walks}
Let $\mathcal D$ be an SD, and let $D$ be the union of the arcs in $\mathcal D$. Note that any point $p\in D$ lies in the interior of exactly one arc $r(p)\in\mathcal D$. A \emph{sliding walk} on $\mathcal D$ is defined as a continuous function $w\colon [0,+\infty)\to D$ such that, for every maximal interval $I\subseteq [0,+\infty)$ with $|r(w(I))|=1$, $I$ is left-closed and right-open, and $w$ restricted to $I$ is a unit-speed regular curve.

That is, a sliding walk continues to move in the same direction at unit speed along the same arc $a\in\mathcal D$ until it reaches one of its endpoints. Then it turns left or right into the (unique) arc hit by $a$, follows it until one of its endpoints, and so on indefinitely.

Since an SD $\mathcal D$ has finitely many arcs, a sliding walk $w$ cannot be an injective function, and there exists a minimum $x$ and a unique interval $[x,x']\subset [0,+\infty)$ such that the restriction of $w$ to $[x,x']$ is a simple closed curve (it is easy to see that either $x=0$ or $w(x)$ is an endpoint of an arc of $\mathcal D$). This curve separates the unit sphere into a \emph{right-side region} and a \emph{left-side region}. Note that these two regions are well defined and unique for any sliding walk.

A \emph{right-handed walk} (resp., \emph{left-handed walk}) is a sliding walk that always turns right (resp., left) upon hitting a new arc.

\begin{observation}\label{yobs:right}
The right-side region (resp., left-side region) of any right-handed walk (resp., left-handed walk) on an SD $\mathcal D$ coincides with the eye of a clockwise swirl (resp., counterclockwise swirl) of $\mathcal D$.\qed
\end{observation}

The following lemma is illustrated in \cref{fig:1a}; a complete proof is in \cref{asec:walks}.

\begin{lemma}\label{ylemma:right}
The right-side region (resp., left-side region) of any sliding walk on an SD $\mathcal D$ contains the eye of a clockwise swirl (resp., counterclockwise swirl) of $\mathcal D$.\qed
\end{lemma}

If $p$ is a point on the unit sphere, a \emph{clockwise walk} (resp., \emph{counterclockwise walk}) with \emph{fulcrum} $p$ is a sliding walk that moves clockwise (resp., counterclockwise) around $p$ in a weakly monotonic fashion, without touching $p$ or the point $p'$ antipodal to $p$.\footnote{In~\cite{spherical2}, such a walk is called ``monotonic walk'', and $p$ is called ``pole''. We adopted a different terminology in this paper to avoid confusion with the poles of $k$-oriented SDs.} In \cref{fig:1b}, the initial steps of a clockwise walk with fulcrum $p$ are shown.

We are ready to settle the open problem stated in~\cite[Conjecture~2]{spherical1}, \cite[Conjecture~22]{spherical2}, and~\cite[Section~6]{mini}. In fact, we prove it not only for all SODs, but more generally for all SDs.

\begin{theorem}\label{yt:22swirls}
Any SD has at least two clockwise swirls and two counterclockwise swirls.
\end{theorem}
\begin{proof}
Let $\mathcal D$ be an SD; it suffices to prove that $\mathcal D$ has two clockwise swirls. Let $S$ be any clockwise swirl of $\mathcal D$ (found, for instance, as in \cref{yobs:right}), and let $E$ be the eye of $S$. Let $p$ be any point in the interior of $E$, and let $w$ be a counterclockwise walk with fulcrum $p$ starting from the boundary of $E$.

Since $S$ is a clockwise swirl, whenever $w$ is in the interior of an arc $a\in S$, it moves toward the endpoint of $a$ not on $E$. In particular, $w$ never reaches the interior of $E$, and therefore the left-side region of $w$ contains $E$. By \cref{ylemma:right}, the right-side region of $w$ contains the eye of a clockwise swirl, which cannot be $E$. Hence, $\mathcal D$ has a clockwise swirl other than $S$.
\end{proof}

\subsection{Swirl Adjacency}\label{sec:swirl}
The missing proofs of the results in this section are found in \cref{asec:swirl}.

Recall from \cref{sec:previous} that the swirl graph of any SOD is simple, i.e., any two swirls of an SOD can have at most one arc in common. Moreover, any arc can contribute to at most two swirls. On the contrary, the swirl graph of a general SD may be a non-simple multigraph, since two swirls in an SD may share more than one arc, as shown in \cref{fig:adjacency}. Also, a single arc of an SD can contribute to up to four different swirls (one on each side of each endpoint).

There are essentially two different ways in which two swirls may share multiple arcs. The first configuration involves two \emph{contiguous swirls} whose eyes are adjacent along an arc $a$ of the SD and also share an endpoint of $a$. An example is given by the yellow and the green swirl in \cref{fig:adjacency}.

\begin{proposition}\label{xp:contiguous}
In any SD, pairs of contiguous swirls share exactly two arcs.\qed
\end{proposition}

Even if two swirls are not contiguous, they may still share multiple arcs: an example is given by the yellow and purple swirls in \cref{fig:adjacency}.

\begin{proposition}\label{yc:noncontiguous}
In any SD, a swirl of degree $d$ may share at most $\lfloor d/2\rfloor$ arcs with the same non-contiguous swirl.\qed
\end{proposition}

Two swirls are said to be \emph{discordant} if one of them is clockwise and the other is counterclockwise; they are \emph{concordant} otherwise.

\begin{proposition}\label{yp:adjswirl1}
In an SD, let $S_1$ and $S_2$ be two swirls that share more than one arc. The following statements are equivalent: (i) $S_1$ and $S_2$ are not contiguous; (ii) $S_1$ and $S_2$ are concordant; (iii) the eyes of $S_1$ and $S_2$ have antipodal interior points.\qed
\end{proposition}

\begin{figure*}
\centering
\begin{subfigure}{0.49\textwidth}
    \includegraphics[scale=0.5]{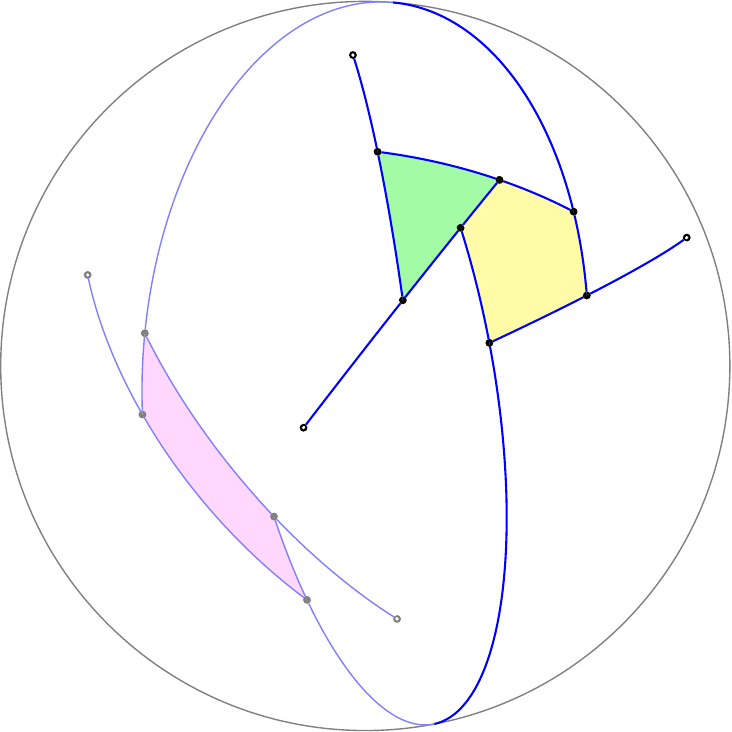}
    \subcaption{}
    \label{fig:adjacency}
\end{subfigure}
\hfill
\begin{subfigure}{0.49\textwidth}
    \includegraphics[scale=0.5]{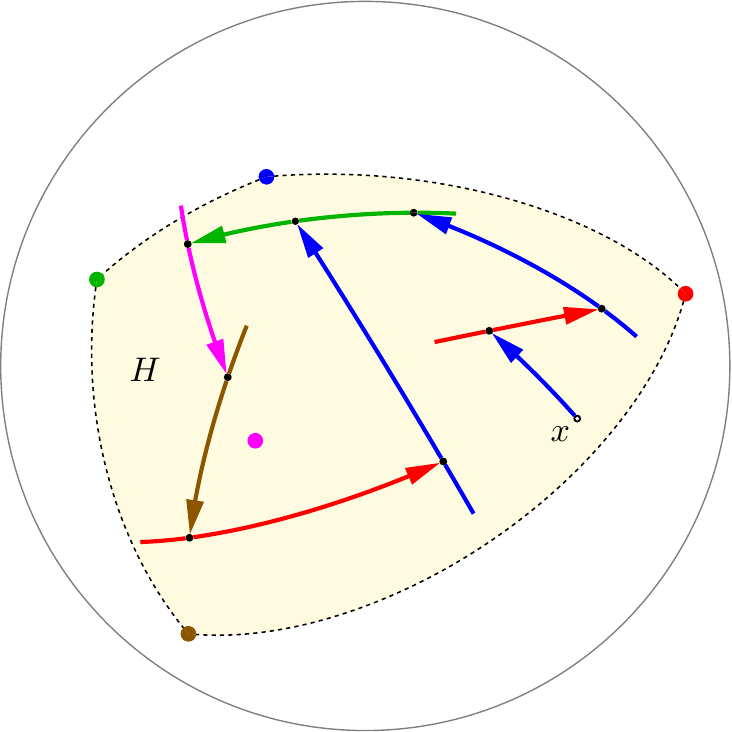}
    \subcaption{}
    \label{fig:walk.a}
\end{subfigure}
\caption{(a) The swirl in yellow shares two arcs with the contiguous swirl in green and two arcs with the non-contiguous swirl in purple. (b) An attractor hull and a sliding walk within its interior.}
\label{fig:prelim}
\end{figure*}

\subsection{Attractors}\label{sec:attractors}
The missing proofs of the results in this section are found in \cref{asec:attractors}.

\begin{definition}
An \emph{attractor} of a $k$-oriented SD is a set of $k$ points, no two of which are antipodal, chosen among its poles and anti-poles. An \emph{attractor hull} is the spherical convex hull of an attractor.
\end{definition}

Since no two poles are antipodal, a $k$-oriented SD has exactly $2^k$ distinct attractors.

\begin{observation}\label{yo:attractor}
Given any attractor $A$ of a $k$-oriented SD $\mathcal D$, each arc of $\mathcal D$ has a unique vanishing point in $A$, and is collinear with it.\qed
\end{observation}

\begin{proposition}\label{yp:noncoll}
The $k$ poles of a $k$-oriented SD cannot be all collinear. Thus, all attractor hulls have non-empty interiors.\qed
\end{proposition}

To state the following results, we have to define the terms \emph{intersect}, \emph{cross}, and \emph{thrust}. By \emph{intersect} we simply mean ``have a non-empty intersection''. Thus, two arcs intersect even if they only share an endpoint. Instead, an arc $a$ \emph{crosses} a curve $\gamma$ when their intersection includes an isolated point $x$ internal to $a$, such that any neighborhood of $x$ contains points of $\gamma$ lying on both sides of $a$. Let $R$ be a spherical polygon contained in the interior of a hemisphere. We say that an arc $a$ \emph{thrusts} the boundary of $R$ at a point $x$ if $a$ intersects the boundary of $R$ at $x$, as well as the interior of $R$.

\begin{proposition}\label{yp:2cross}
If an arc of a $k$-oriented SD intersects the interior of an attractor hull $H$, it intersects the boundary of $H$ in at most one point.\qed
\end{proposition}

If $\mathcal D$ is an SD and $R$ is a region of the unit sphere, we say that two points $x$ and $y$ are \emph{$R$-connected} (with respect to $\mathcal D$) if there is a path with endpoints $x$ and $y$ that follows the arcs of $\mathcal D$ while remaining within $R$. If such a path is internal to $R$ (except for its endpoints $x$ and $y$, which may be on the boundary of $R$), then $x$ and $y$ are \emph{internally $R$-connected}.

\begin{proposition}\label{yp:1inters}
If an arc of a $k$-oriented SD $\mathcal D$ intersects the boundary of an attractor hull $H$ at a point $x$, then there is an arc of $\mathcal D$ that crosses the boundary of $H$ at a point $y$, such that $x$ and $y$ are $H$-connected (with respect to $\mathcal D$) and lie on a same edge of $H$. Moreover, if $x\neq y$, then $y$ is not a vertex of $H$.\qed
\end{proposition}

We say that an attractor hull is \emph{total} if it coincides with the entire unit sphere. If an attractor hull is not total, then it is contained in the interior of a hemisphere.

The following result justifies the choice of the term ``attractor''.

\begin{proposition}\label{yp:attractor}
Let $H$ be an attractor hull of a $k$-oriented SD $\mathcal D$. Then, there exist arcs of $\mathcal D$ that intersect the interior of $H$. Moreover, any point of intersection between an arc of $\mathcal D$ and the interior of $H$ is internally $H$-connected with the vertices of the eye of a swirl of $\mathcal D$.
\end{proposition}
\begin{proof}
Let $A$ be an attractor of $\mathcal D$, and let $H$ be the spherical convex hull of $A$. If $H$ is total, there is nothing to prove, because $\mathcal D$ is connected and it has at least a swirl (see \cref{sec:previous}).

Otherwise, assume for a contradiction that no arc of $\mathcal D$ intersects $H$. Then, $H$ is contained in the interior of a spherically convex tile $T$ (see \cref{sec:previous}). If $a$ is any arc of $\mathcal D$ bounding $T$, then $T$ lies on one side of the great circle $\gamma$ containing $a$, and therefore $\gamma$ does not intersect $H$. Hence, $a$ is not collinear with any point in $A\subset H$, contradicting \cref{yo:attractor}. Consequently, there is an arc of $\mathcal D$ that intersects $H$, and by \cref{yp:1inters} there is also an arc that intersects the interior of $H$ (note that $H$ has an interior, due to \cref{yp:noncoll}).

Now, let $x$ be any point of intersection between an arc of $\mathcal D$ and the interior of $H$, and let $w$ be a sliding walk that starts from $x$ and follows each arc in the direction of its vanishing point in $A$, which exists due to \cref{yo:attractor}. An example of such a sliding walk is shown in \cref{fig:walk.a}. Since $H$ is convex and $w$ always moves toward a point of $A\subset H$ without ever reaching it, we conclude that $w$ never leaves the interior of $H$. Thus, either the left-side or the right-side region of $w$ is entirely contained in the interior of $H$. In turn, this region contains the eye of a swirl, due to \cref{ylemma:right}.
\end{proof}

\begin{corollary}\label{yc:attractoreye}
In any $k$-oriented SD $\mathcal D$, the interior of any attractor hull contains the eye of a swirl of $\mathcal D$.\qed
\end{corollary}

\begin{proposition}\label{yp:3intercirc}
Given an SD $\mathcal D$, any great circle on the unit sphere is crossed by at least three arcs of $\mathcal D$.\qed
\end{proposition}


\begin{proposition}\label{yp:3interhull}
Let $H$ be a non-total attractor hull of a $k$-oriented SD $\mathcal D$, and let $x$ be a point of intersection between an arc of $\mathcal D$ and the interior of $H$. Then, there are three distinct arcs of $\mathcal D$ that thrust the boundary of $H$ at three distinct points, not all lying on the same edge of $H$, all of which are internally $H$-connected with $x$.\qed
\end{proposition}

The previous result cannot be improved, as there are $k$-oriented SDs whose arcs thrust the boundary of an attractor hull at exactly three points, two of which lie on the same edge.


\section{Minimal Arrangements}\label{sec:6}

In this section, we will derive all the results listed in \cref{tab:1}.

\begin{proposition}\label{p:no12orien}
There are no $1$-oriented or $2$-oriented SDs.
\end{proposition}
\begin{proof}
The poles of such SDs would be collinear, which contradicts \cref{yp:noncoll}.
\end{proof}

An \emph{attractor partition} for a $k$-oriented SD $\mathcal D$ is a collection of internally disjoint non-total attractor hulls of $\mathcal D$ that collectively cover the unit sphere. The \emph{vertices} and \emph{edges} of an attractor partition are the vertices and edges of the attractor hulls that constitute it.

\begin{lemma}\label{l:attpart}
Let $\mathcal D$ be a $k$-oriented SD with an attractor partition $\mathcal S$ consisting of $m$ attractor hulls. Then, $\mathcal D$ has at least $m$ non-contiguous swirls, each of which has an eye in the interior of a different attractor hull of $\mathcal S$. Moreover, $\mathcal D$ has at least $\lceil 3m/2\rceil$ arcs that thrust edges of $\mathcal S$, and no arc of $\mathcal D$ thrusts the boundaries of more than two attractor hulls.
\end{lemma}
\begin{proof}
Since the attractor hulls constituting $\mathcal S$ are internally disjoint, \cref{yp:attractor} guarantees that each of their interiors contains a swirl of $\mathcal D$. Such swirls are distinct and non-contiguous, because their eyes do not share any boundary points.

Moreover, \cref{yp:3interhull} ensures that at least three arcs of $\mathcal D$ thrust the boundary of each attractor hull of $\mathcal S$. Also, no arc of $\mathcal D$ thrusts the boundary of any attractor hull of $\mathcal S$ at more than one point, due to \cref{yp:2cross}. Hence, no arc of $\mathcal D$ thrusts the boundaries of more than two attractor hulls of $\mathcal S$. We conclude that $\mathcal D$ has at least $\lceil 3m/2\rceil$ arcs that thrust boundaries of attractor hulls of $\mathcal S$.
\end{proof}

\subsection{3-Oriented SDs}\label{sec:6.1}

Observe that no $3$-oriented SDs are degenerate, otherwise their three poles would be collinear, contradicting \cref{yp:noncoll}. In fact, there is essentially one possible configuration for the poles of a $3$-oriented SD, and the set of its $2^3=8$ attractor hulls constitutes an attractor partition. We call each attractor hull of a $3$-oriented SD an \emph{octant} (see \cref{fig:222}).

\begin{theorem}\label{t:3orien}
Any $3$-oriented SD has at least eight swirls and $12$ arcs. Moreover, there are matching examples that are SODs.
\end{theorem}
\begin{proof}
Since any such SD has an attractor partition consisting of $m=8$ octants, it follows from \cref{l:attpart} that it has at least $m=8$ swirls and $3m/2=12$ arcs. A $3$-oriented SOD with exactly eight swirls and $12$ arcs is shown in \cref{fig:222}.
\end{proof}

\begin{figure*}
\centering
\begin{subfigure}{0.49\textwidth}
    \includegraphics[scale=0.5]{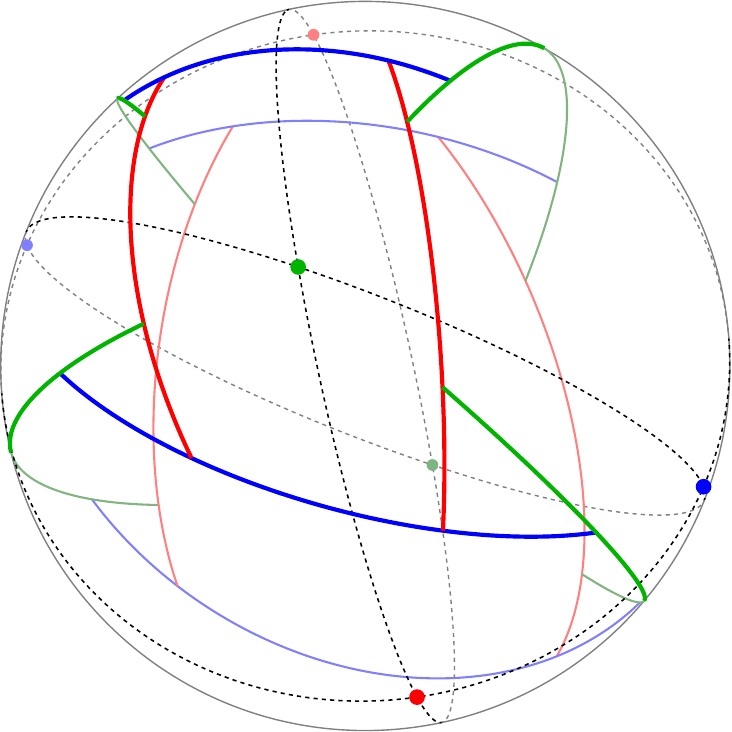}
    \subcaption{}
    \label{fig:222}
\end{subfigure}
\hfill
\begin{subfigure}{0.49\textwidth}
    \includegraphics[scale=0.5]{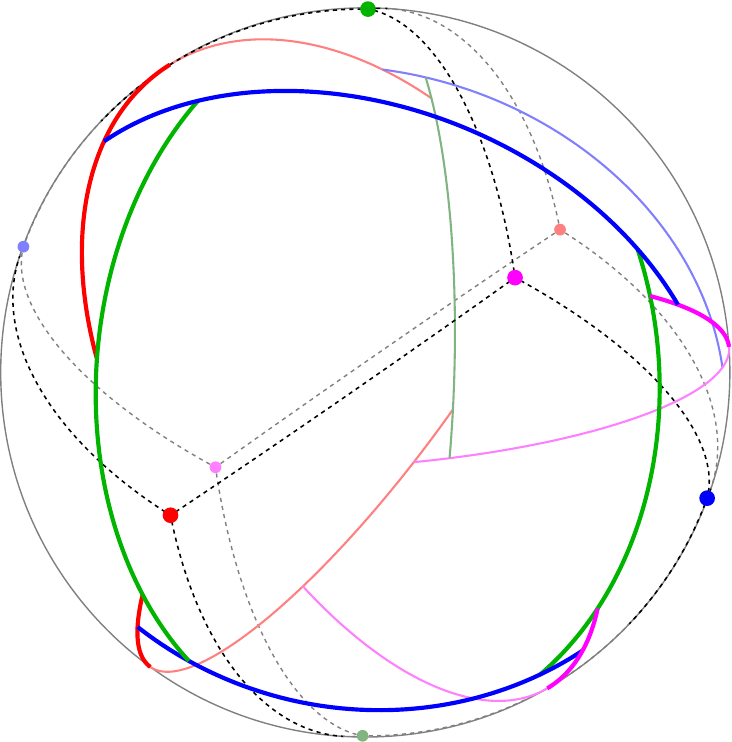}
    \subcaption{}
    \label{fig:3333}
\end{subfigure}
\caption{(a) A $3$-oriented SD with exactly eight swirls and $12$ arcs. (b) A non-degenerate $4$-oriented SOD with exactly six swirls and ten arcs. Each arc has the same color as its vanishing points. Dashed lines mark the boundaries of octants and sextants, respectively.}
\label{fig:triquad}
\end{figure*}

\subsection{4-Oriented SDs}\label{sec:6.2}
There are essentially two possible configurations of four poles on a sphere. Starting from three poles determining eight octants, a fourth pole can be placed either in the interior of an octant, giving rise to a non-degenerate configuration with alignment $(3,3,3,3)$, or on the boundary between two octants, giving rise to a degenerate configuration with alignment $(2,2,2,3)$. Due to \cref{yp:noncoll}, there are no other configurations.

Every non-degenerate $4$-oriented SD has an attractor partition consisting of six strictly convex quadrilaterals, which we call \emph{sextants}.\footnote{The term ``sextant'' traditionally refers to a navigation instrument. Its name is derived from the fact that a sextant's arc covers $60$ degrees, which is one sixth of a circle. In this paper, we use the term to indicate each of six regions that partition a sphere.} An example is in \cref{fig:3333}.

\begin{theorem}\label{t:4nondeg}
Any non-degenerate $4$-oriented SD has at least six swirls and nine arcs; if it is an SOD, it has at least ten arcs. The bounds are tight.
\end{theorem}
\begin{proof}
Since any such SD has an attractor partition consisting of $m=6$ sextants, it follows from \cref{l:attpart} that it has at least $m=6$ non-contiguous swirls and $3m/2=9$ arcs. These six swirls, one per sextant, are called \emph{sextant swirls}.

It remains to prove that any non-degenerate $4$-oriented SOD $\mathcal D$ has at least ten arcs. Let $G$ be the subgraph of the swirl graph of $\mathcal D$ induced by its sextant swirls. As we recalled in \cref{sec:previous}, the swirl graph of $\mathcal D$ is simple, planar, and bipartite; therefore, so is $G$. We will prove that $G$ has at most eight edges.

Recall that $G$ has six vertices. If the partite sets of $G$ have $1$ and $5$ vertices, then $G$ has at most five edges. If the partite sets of $G$ have $2$ and $4$ vertices, then $G$ has at most eight edges. Finally, if the partite sets of $G$ have $3$ and $3$ vertices, then $G$ has at most eight edges, because the complete bipartite graph $K_{3,3}$ is not planar. Thus, $G$ has at most eight edges.

Since a swirl involves at least three arcs, and each arc of an SOD contributes to at most two swirls, the arcs of $\mathcal D$ that are involved in sextant swirls are at least $6\times 3-8=10$, and hence $\mathcal D$ has at least ten arcs.

A non-degenerate $4$-oriented SD with exactly six swirls and nine arcs is obtained by perturbing the poles of the SD in \cref{fig:2223SD}. A non-degenerate $4$-oriented SOD with exactly six swirls and ten arcs is shown in \cref{fig:3333}.
\end{proof}

As for degenerate $4$-oriented SDs, we treat them as special cases of $k$-oriented SDs with alignment $(2,2,\dots,2,k-1)$. It turns out that all of these SDs have the same minimum number of swirls and arcs, for any $k\geq 4$. However, since they do not have an attractor partition for $k>4$, analyzing them will require an ad-hoc technique.

\begin{figure*}
\centering
\begin{subfigure}{0.49\textwidth}
    \includegraphics[scale=0.5]{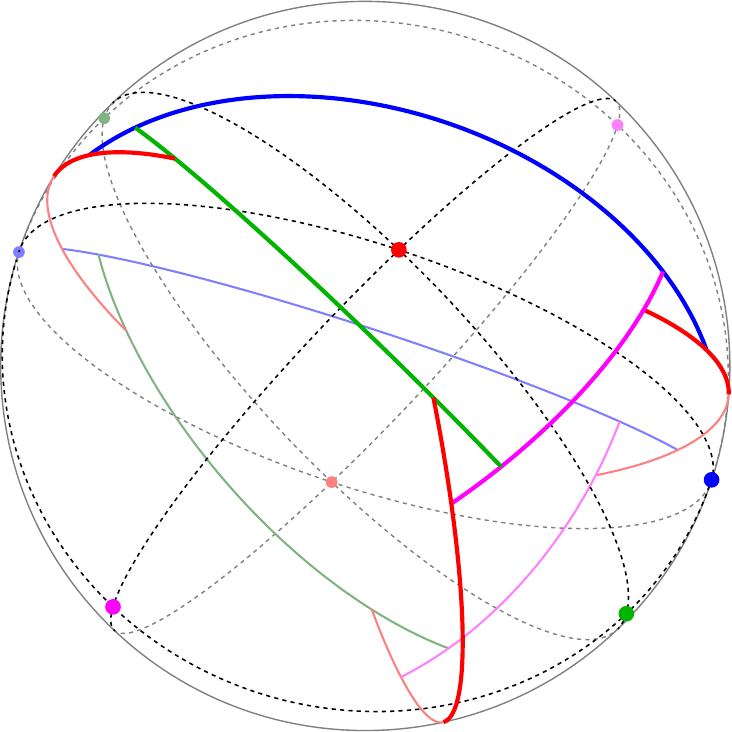}
    \subcaption{}
    \label{fig:2223SD}
\end{subfigure}
\hfill
\begin{subfigure}{0.49\textwidth}
    \includegraphics[scale=0.5]{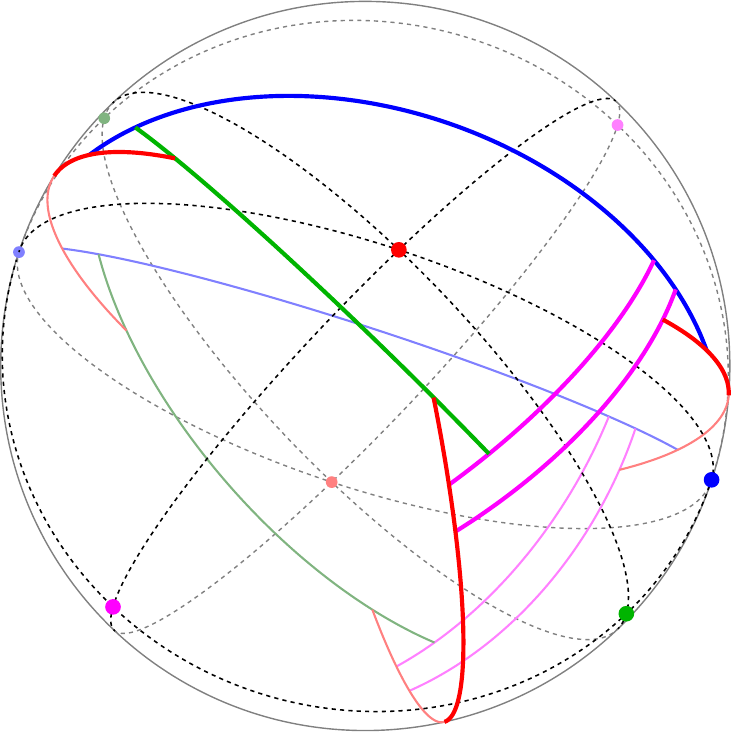}
    \subcaption{}
    \label{fig:2223SOD}
\end{subfigure}
\caption{(a) A $4$-oriented SD with alignment $(2,2,2,3)$ having exactly six swirls and nine arcs. (b) A $4$-oriented SOD with alignment $(2,2,2,3)$ having exactly six swirls and 11 arcs.}
\label{fig:quaddeg}
\end{figure*}

\begin{theorem}\label{t:222k}
If $k\geq 4$, any $k$-oriented SD with alignment $(2,2,\dots,2,k-1)$ has at least six swirls and nine arcs; if it is an SOD, it has at least 11 arcs. The bounds are tight.
\end{theorem}
\begin{proof}
All poles lie on a same great circle $\gamma$, except for one pole $p$. Due to \cref{yp:3intercirc}, $\gamma$ is crossed by at least three arcs, which must have $p$ as a vanishing point. Let $H$ and $H'$ be the two hemispheres bounded by $\gamma$, with $p\in H$. Let $w$ be a clockwise walk with fulcrum $p$ starting from the interior of $H$, with the property that the arcs with vanishing point $p$ are followed in the direction of $p$ (as opposed to the anti-pole $-p$). Clearly, $w$ always remains within $H$, and it takes at least three distinct arcs with vanishing points on $\gamma$ to do a complete turn around $p$. Note that such arcs lie entirely in the interior of $H$. Similarly, there are at least three arcs lying in the interior of $H'$, which yields nine arcs in total.

There are $2k-2$ distinct attractor hulls in $H$. By \cref{yp:attractor}, each of them contains the eye of a swirl; let $E$ be one of these eyes. Let $\gamma'$ be a great circle through $p$ that intersects the interior of $E$. Note that $\gamma'$ separates $H$ into two parts, each of which contains an attractor hull, and hence the eye of a swirl distinct from $E$. Thus, there are at least three swirls whose eyes are in $H$; for the same reason, there are at least three swirls whose eyes are in $H'$.

Let $G$ be the subgraph of the swirl graph induced by these six swirls; we will prove that, in an SOD, $G$ has at most seven arcs. Observe that, in any of these swirls, at most two arcs may have $p$ as a vanishing point (otherwise the eye would not be spherically convex), and at most one of them may cross $\gamma$. Recall from \cref{sec:previous} that $G$ is simple and bipartite. Hence, at most two arcs may be shared among the three swirls whose eyes are in $H$, and at most two arcs may be shared among the other three swirls. Thus, $G$ has at most three arcs crossing $\gamma$ and at most four other arcs. Therefore, the six swirls involve at least $6\times 3-7=11$ arcs.

Matching examples for $k=4$ are shown in \cref{fig:quaddeg}; adding ``dummy poles'' on the great circle that already contains three of them yields examples for all $k>4$, as well.
\end{proof}

\begin{figure*}
\centering
\begin{subfigure}{0.49\textwidth}
    \includegraphics[scale=0.5]{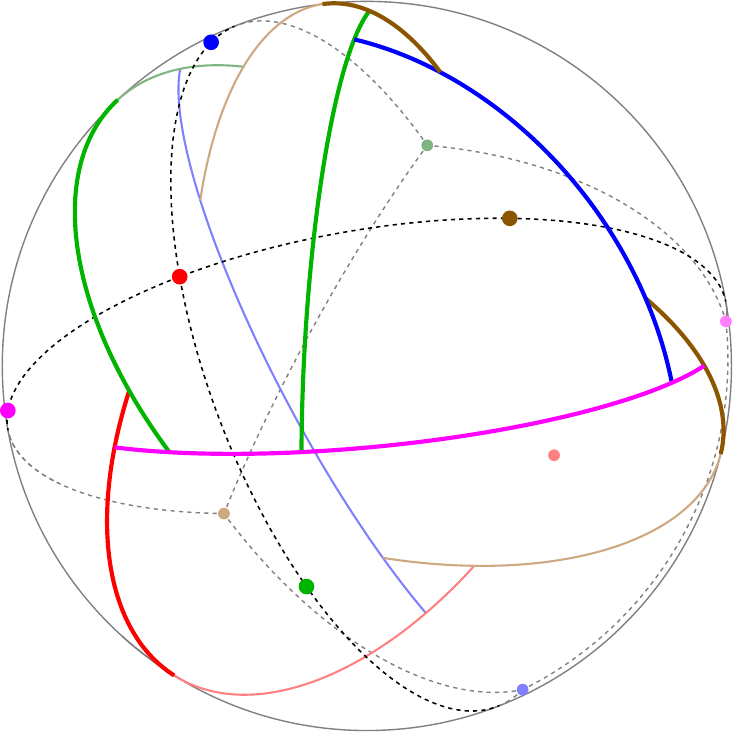}
    \subcaption{}
    \label{fig:23333SD}
\end{subfigure}
\hfill
\begin{subfigure}{0.49\textwidth}
    \includegraphics[scale=0.5]{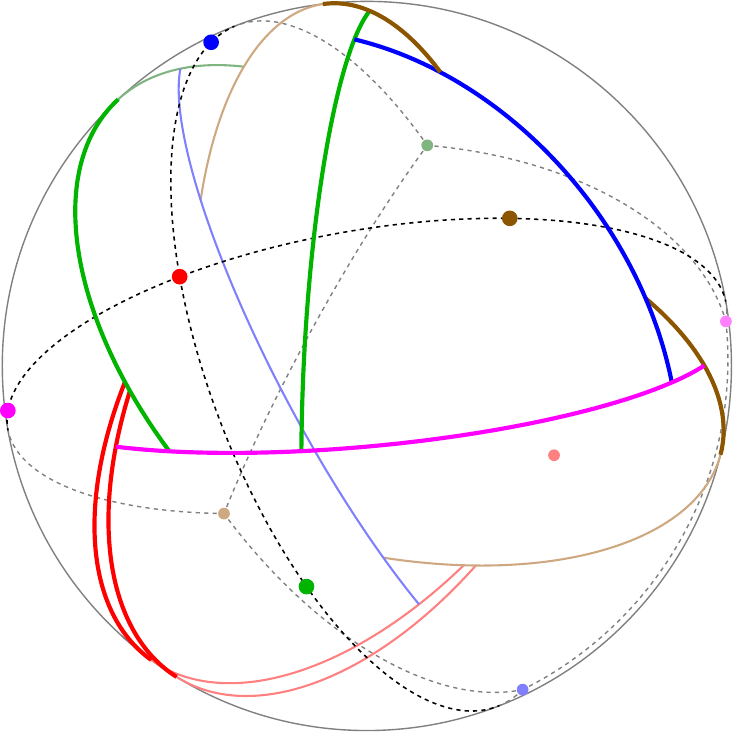}
    \subcaption{}
    \label{fig:23333SOD}
\end{subfigure}
\caption{(a) A $5$-oriented SD with alignment $(2,3,3,3,3)$ having exactly five swirls and eight arcs. (b) A $5$-oriented SOD with alignment $(2,3,3,3,3)$ having exactly five swirls and nine arcs.}
\label{fig:23333}
\end{figure*}

\subsection{5-Oriented SDs}\label{sec:6.3}
There are essentially four distinct configurations of five poles. By \cref{yp:noncoll}, the poles cannot all be collinear; nonetheless, four of them may lie on a same great circle, giving rise to a configuration with alignment $(2,2,2,2,4)$, which was already discussed in \cref{t:222k}.

If three poles lie on a same great circle $\gamma$, we have two possible alignments: $(2,3,3,3,3)$ if the two remaining poles are collinear with a pole on $\gamma$ (see \cref{fig:23333zones}), and $(3,3,3,4,4)$ otherwise. Each of these alignments completely determines the combinatorics of the configuration.

Finally, we have the non-degenerate alignment $(4,4,4,4,4)$. To see why all configurations of five poles with this alignment are equivalent, consider four poles with alignment $(3,3,3,3)$. The great circles through pairs of these poles partition each sextant into four spherical triangles. Since these 24 triangles are all equivalent, placing a fifth pole in the interior of any one of them yields combinatorially equivalent configurations with alignment $(4,4,4,4,4)$.

\begin{figure*}
\centering
\begin{subfigure}{0.49\textwidth}
    \includegraphics[scale=0.5]{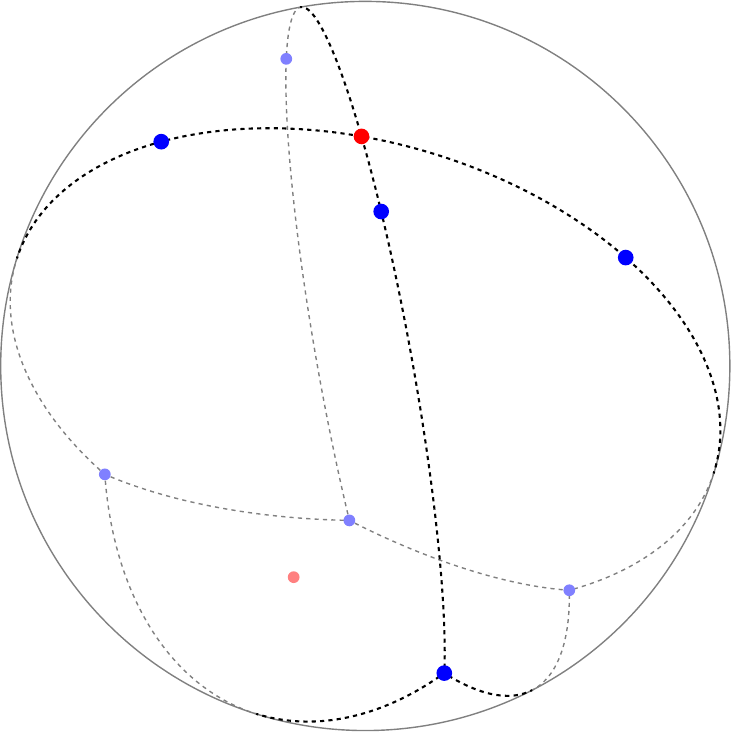}
    \subcaption{}
    \label{fig:23333zones}
\end{subfigure}
\hfill
\begin{subfigure}{0.49\textwidth}
    \includegraphics[scale=0.5]{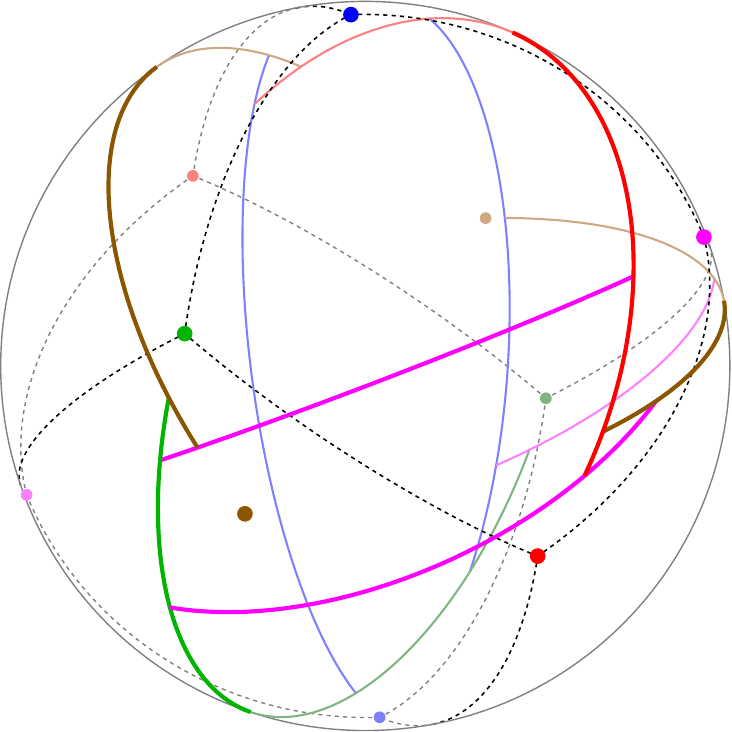}
    \subcaption{}
    \label{fig:44444}
\end{subfigure}
\caption{(a) A configuration of five poles with alignment $(2,3,3,3,3)$ and an attractor partition. (b) A non-degenerate $5$-oriented SD with exactly four swirls and nine arcs.}
\label{fig:various}
\end{figure*}

\begin{theorem}\label{t:23333}
Any $5$-oriented SD with alignment $(2,3,3,3,3)$ has at least five swirls. Moreover, there are SDs with alignment $(2,3,3,3,3)$ that have exactly five swirls and eight arcs, and SODs with alignment $(2,3,3,3,3)$ that have exactly five swirls and nine arcs.
\end{theorem}
\begin{proof}
As shown in \cref{fig:23333zones}, any SD with alignment $(2,3,3,3,3)$ has an attractor partition consisting of $m=5$ attractor hulls. It follows from \cref{l:attpart} that it has at least $m=5$ non-contiguous swirls. An SD and an SOD with the desired properties are found in \cref{fig:23333}.
\end{proof}

By perturbing the red poles in \cref{fig:23333}, one obtains configurations of SDs and SODs with alignments $(3,3,3,4,4)$ and $(4,4,4,4,4)$ with the same characteristics. In the latter case, however, we can actually construct SDs and SODs with only four swirls (which is minimum).

\begin{figure*}
\centering
\begin{subfigure}{0.49\textwidth}
    \includegraphics[scale=0.6]{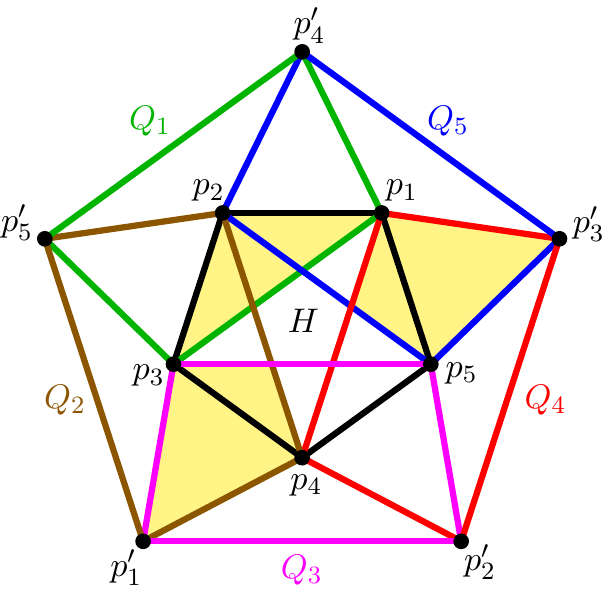}
    \subcaption{}
    \label{fig:kx1}
\end{subfigure}
\hfill
\begin{subfigure}{0.49\textwidth}
    \includegraphics[scale=0.6]{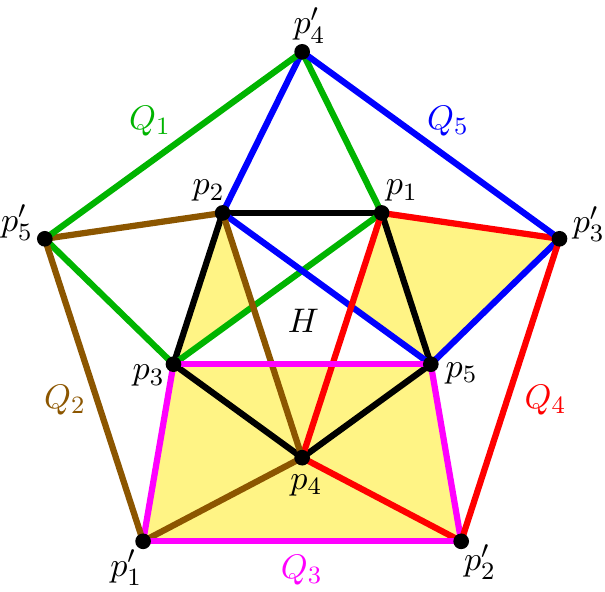}
    \subcaption{}
    \label{fig:kx2}
\end{subfigure}
\caption{Sketch of the two cases in the proof of \cref{l:5nondeg}.}
\label{fig:kx}
\end{figure*}

\begin{figure*}
\centering
\begin{subfigure}{0.49\textwidth}
    \includegraphics[scale=0.5]{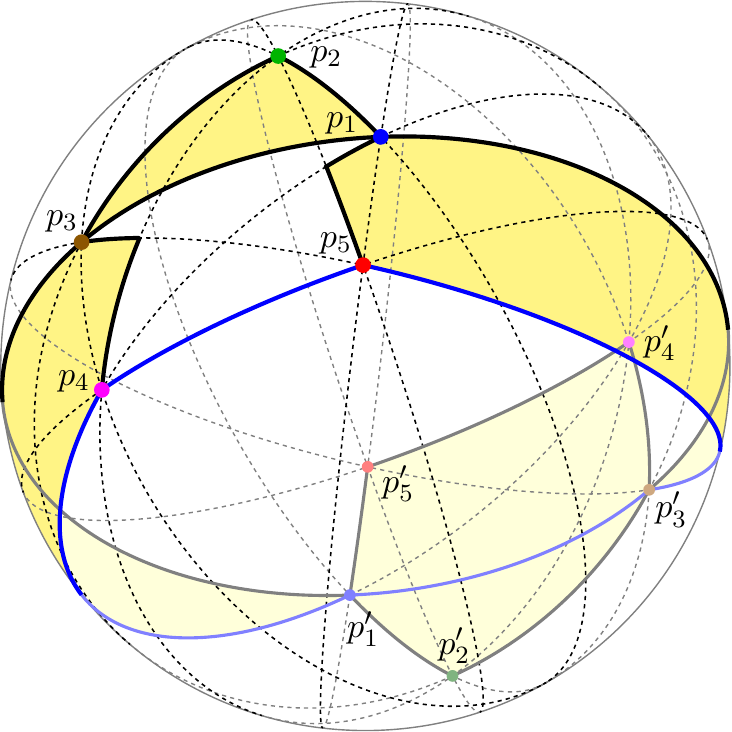}
    \subcaption{}
    \label{fig:k1.a}
\end{subfigure}
\hfill
\begin{subfigure}{0.49\textwidth}
    \includegraphics[scale=0.5]{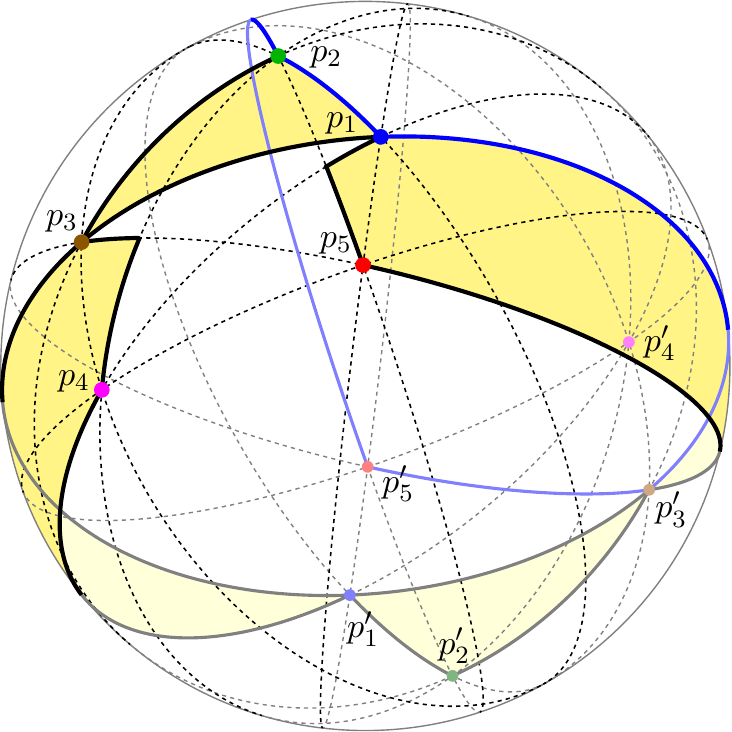}
    \subcaption{}
    \label{fig:k1.b}
\end{subfigure}
\caption{Case~1 in the proof of \cref{l:5nondeg}.}
\label{fig:k1}
\end{figure*}

\begin{figure*}
\centering
\begin{subfigure}{0.49\textwidth}
    \includegraphics[scale=0.5]{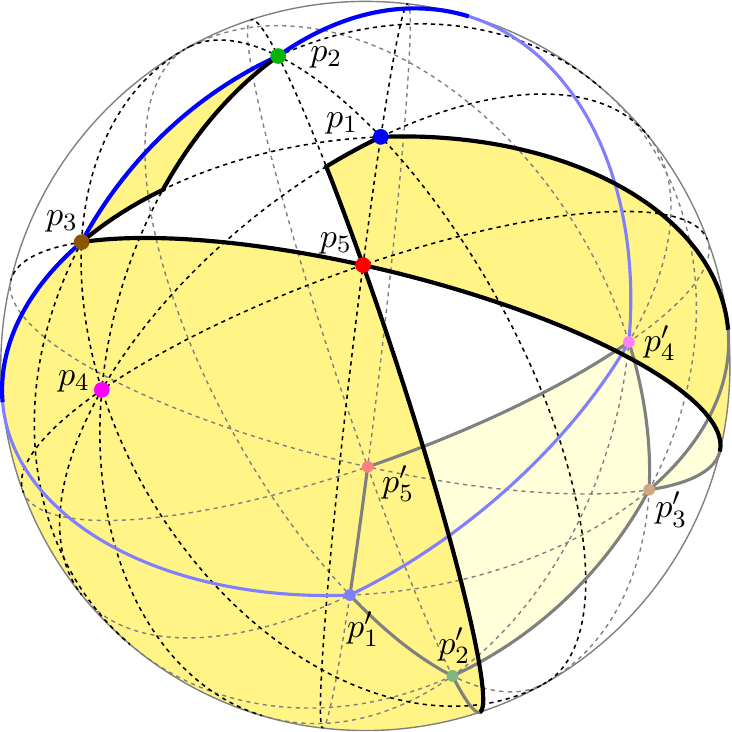}
    \subcaption{}
    \label{fig:k2.a}
\end{subfigure}
\hfill
\begin{subfigure}{0.49\textwidth}
    \includegraphics[scale=0.5]{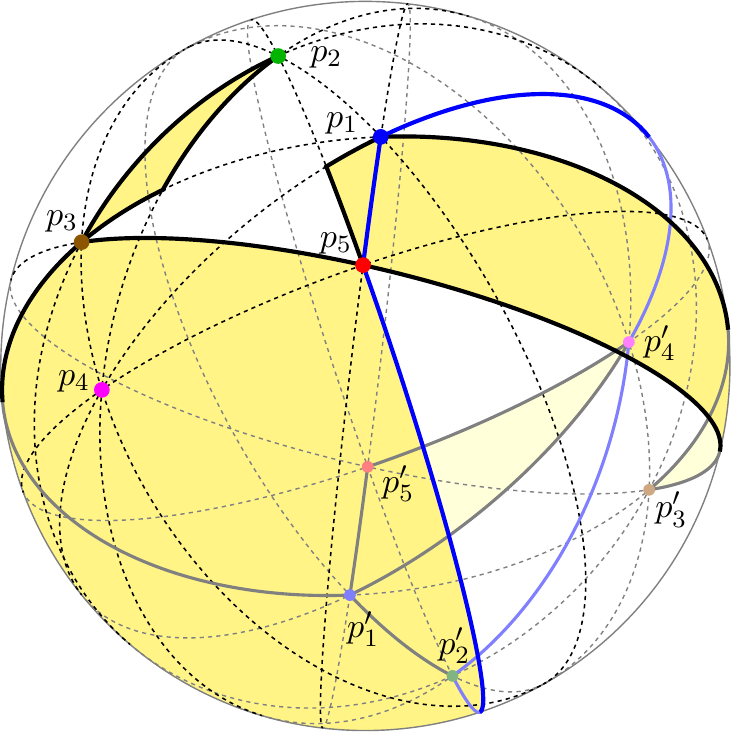}
    \subcaption{}
    \label{fig:k2.b}
\end{subfigure}

\begin{subfigure}{0.49\textwidth}
    \includegraphics[scale=0.5]{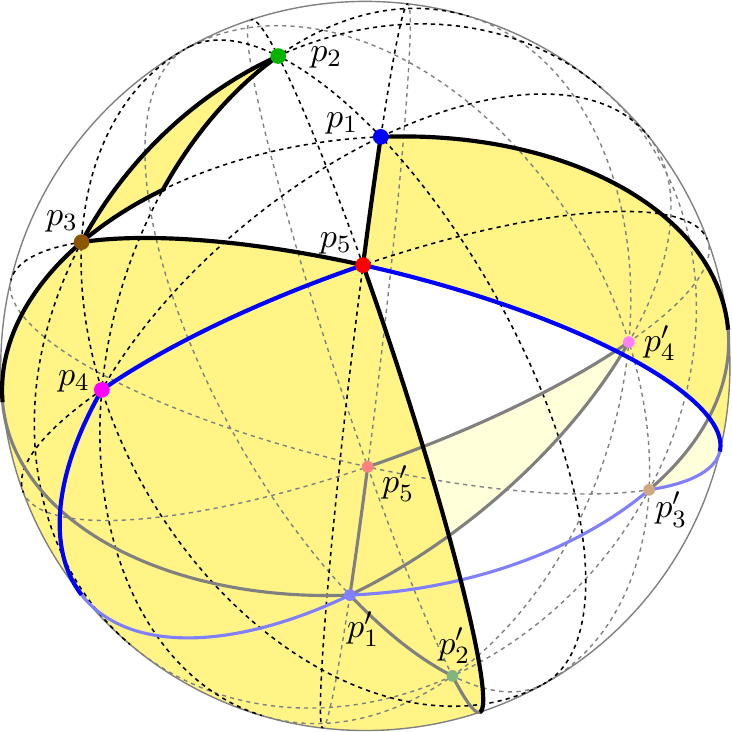}
    \subcaption{}
    \label{fig:k2.c}
\end{subfigure}
\hfill
\begin{subfigure}{0.49\textwidth}
    \includegraphics[scale=0.5]{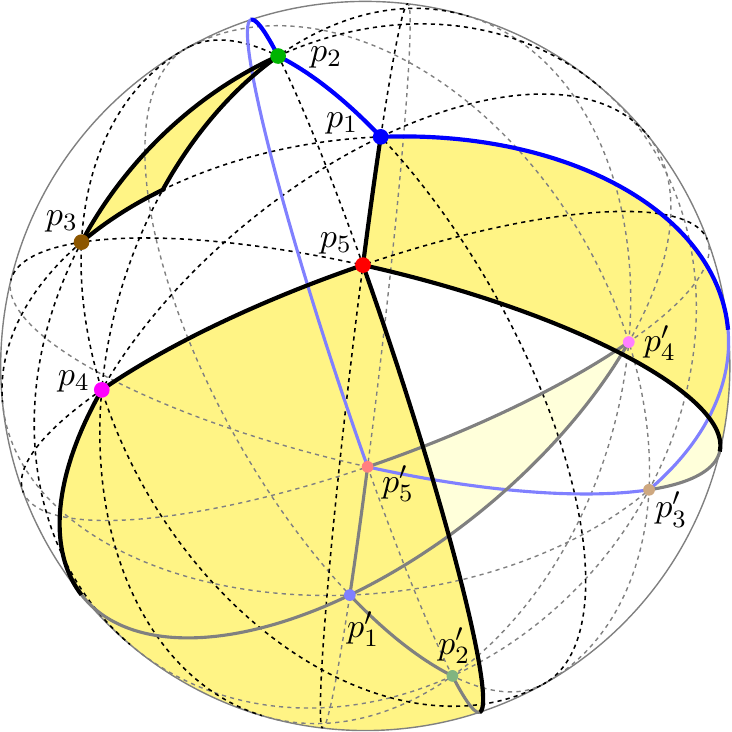}
    \subcaption{}
    \label{fig:k2.d}
\end{subfigure}

\begin{subfigure}{0.49\textwidth}
    \includegraphics[scale=0.5]{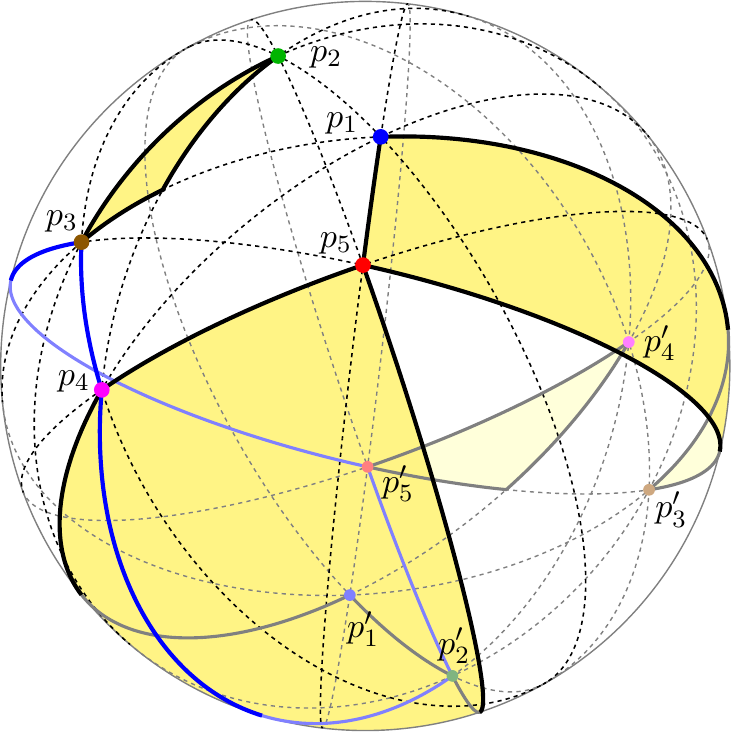}
    \subcaption{}
    \label{fig:k2.e}
\end{subfigure}
\hfill
\begin{subfigure}{0.49\textwidth}
    \includegraphics[scale=0.5]{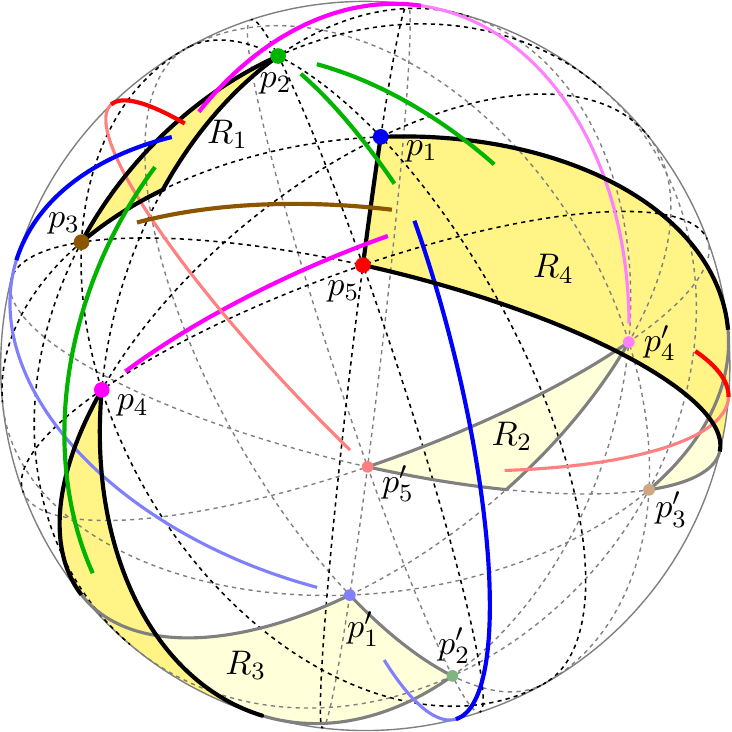}
    \subcaption{}
    \label{fig:k2.f}
\end{subfigure}
\caption{Case~2 in the proof of \cref{l:5nondeg}.}
\label{fig:k2}
\end{figure*}

\begin{lemma}\label{l:5nondeg}
If a $5$-oriented SD has exactly four swirls, then it has at least nine arcs. There are matching non-degenerate examples.
\end{lemma}
\begin{proof}
It is not restrictive to focus on non-degenerate SDs, because the poles of a degenerate SD can be perturbed without altering the number of swirls and arcs, making it non-degenerate.

As argued above, there is essentially a unique non-degenerate configuration of five poles, in the sense that drawing great circles through pairs of such poles yields combinatorially equivalent partitions of the sphere into spherical polygons. Moreover, it is easy to see that there is always a spherically convex pentagonal attractor hull. In fact, without loss of generality, we may assume that the five poles $p_1$, \dots, $p_5$ are the vertices of such an attractor hull $H$; likewise, the anti-poles $p'_1$, \dots, $p'_5$ are also vertices of a spherically convex pentagonal attractor hull $H'$, as shown in \cref{fig:k1}.

Assume that there are four swirls; we will identify four interior-disjoint regions of the sphere where their eyes must be located. By \cref{yc:attractoreye}, every attractor hull contains the eye of a swirl. Hence, $H'$ contains the eye of a swirl: this is the first region. Consider the attractor hull $H$, as well as the five quadrilateral attractor hulls $Q_1$, \dots, $Q_5$, each of which shares an edge with $H'$ and two non-consecutive vertices with $H$. For example, $Q_1$ has vertices $p'_4$, $p'_5$, $p_3$, $p_1$ and contains $p_2$ in its interior, etc., as shown in \cref{fig:kx}. The eyes of the remaining three swirls must be shared among these six attractor hulls.

We will distinguish two cases. In Case~1, no eye lies in the intersection of more than two of these six attractor hulls. Since the $Q_i$'s are all equivalent, we may assume that the three eyes are in the regions $H\cap Q_1$, $Q_2\cap Q_3$, and $Q_4\cap Q_5$, as sketched in \cref{fig:kx1}. Since the interior of the attractor hull $p'_1p'_3p_5p_4$ only intersects the chosen region $H'$ (\cref{fig:k1.a}), there must be an eye in their intersection. However, this implies that the attractor hull $p'_3p'_5p_2p_1$ has no eyes in its interior (\cref{fig:k1.b}), which is a contradiction.

In Case~2, an eye lies in the intersection of three of the six attractor hulls, say $H\cap Q_1\cap Q_2$. Without loss of generality, the other two eyes are in the regions $Q_3$ and $Q_4\cap Q_5$, as sketched in \cref{fig:kx2}. Again, we can progressively reduce these four regions by finding attractor hulls whose interiors intersect only one region, as detailed in \cref{fig:k2}. Finally, we obtain the four regions $R_1$, $R_2$, $R_3$, $R_4$ shown in \cref{fig:k2.f}, where the only pairs of regions that may be connected by an arc are $(R_1,R_3)$, $(R_2,R_4)$, and $(R_3,R_4)$. Indeed, observe that $R_2$ and $R_3$ are equivalent to $R_1$ and $R_4$, respectively. Thus, we only have to verify that no arc can connect $R_1$ with $R_2$ or with $R_4$. Furthermore, due to \cref{yp:adjswirl1}, no two swirls may share more than one arc, since no two regions have antipodal internal points. Hence, there are at most three arcs that are shared between pairs of swirls. We conclude that the arcs contributing to the four swirls are at least $4\times 3-3=9$.

An SD with the desired properties is shown in \cref{fig:44444}.
\end{proof}

\begin{figure*}
\centering
\begin{subfigure}{0.49\textwidth}
    \includegraphics[scale=0.5]{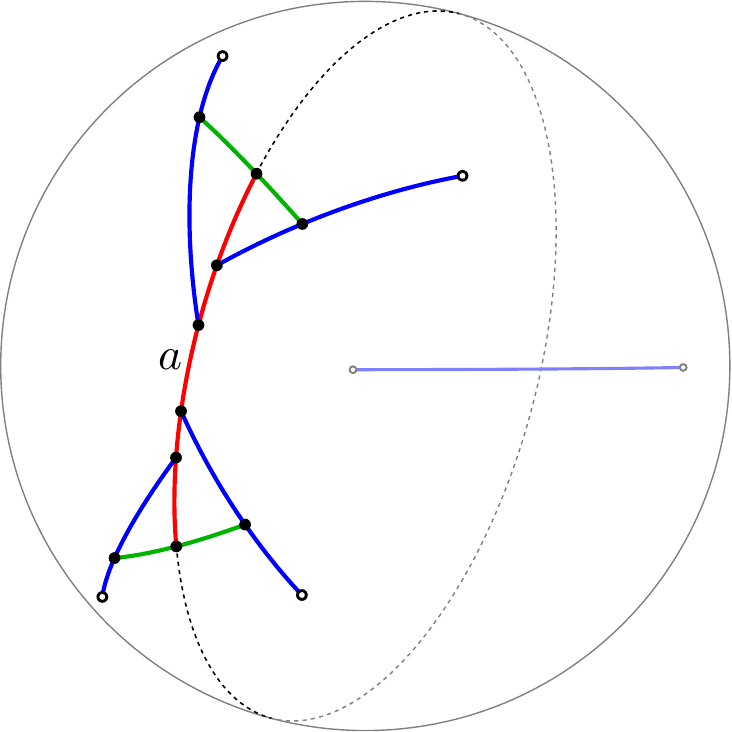}
    \subcaption{}
    \label{fig:8arc}
\end{subfigure}
\hfill
\begin{subfigure}{0.49\textwidth}
    \includegraphics[scale=0.5]{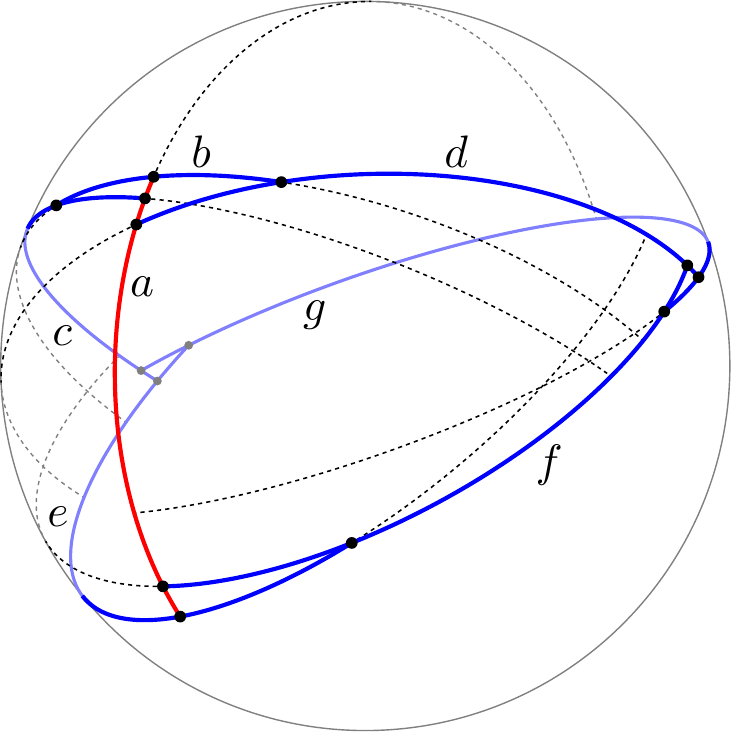}
    \subcaption{}
    \label{fig:7arc}
\end{subfigure}
\caption{(a) If an arc $a$ contributes to four swirls, the SD has at least eight arcs. (b) In an SD with seven arcs, if $a$ contributes to three swirls, only $b$ and $g$ may have a common vanishing point.}
\label{fig:78arc}
\end{figure*}

\begin{proposition}\label{p:poleanti}
In a $k$-oriented SD, if two non-collinear arcs $a$ and $b$ have a common vanishing point, then the great circle containing $a$ does not intersect $b$, and vice versa.
\end{proposition}
\begin{proof}
The vanishing points of $a$ and $b$ are the two intersection points between the great circles containing $a$ and $b$. By definition of $k$-oriented SD, $b$ does not intersect its vanishing points, and thus it cannot intersect the great circle containing $a$, and vice versa.
\end{proof}

\begin{theorem}\label{t:5nondeg}
Any $5$-oriented SD has at least eight arcs; if it is an SOD, it has at least nine arcs.
\end{theorem}
\begin{proof}
Due to \cref{yt:22swirls,l:5nondeg}, we may assume that there are at least five swirls.

Recall from \cref{sec:swirl} that, since swirls may be contiguous, each arc of an SD may contribute to up to four different swirls. If an arc $a$ contributes to four swirls, it must hit two arcs, and it must block (at least) four arcs, all of which are distinct (see \cref{fig:8arc}). Of these seven arcs, only the two hit by $a$ cross the great circle containing $a$. Thus, by \cref{yp:3intercirc}, there must be another arc crossing this great circle, implying that there are at least eight arcs in the SD.

Consider now an SD where an arc $a$ is shared by exactly three swirls (two of which are contiguous), and assume that there are fewer than eight arcs. As shown in \cref{fig:7arc}, two arcs $b$ and $e$ block $a$ (where $b$ contributes to the two contiguous swirls), and three arcs $c$, $d$, $f$ hit $a$. As in the previous case, there must be a seventh arc $g$ crossing the great circle $\gamma$ containing $a$, otherwise the only two arcs crossing $\gamma$ would be $b$ and $e$. Since there are at least five swirls, the arcs $c$, $d$, $e$, $f$, $g$ must form two swirls away from $a$, with eyes located on opposite sides of $\gamma$ (we recalled in \cref{sec:previous} that the interior of any hemisphere contains the eye of a swirl). We argue that only two of these seven arcs may have a common vanishing point. Indeed, by \cref{p:poleanti}, incident arcs cannot have a common vanishing point. Also, the extensions of some arcs necessarily hit other arcs, as the dashed lines in \cref{fig:7arc} illustrate. Again, \cref{p:poleanti} implies that such pairs of arcs do not have a common vanishing point. In fact, the only two arcs that may have the same vanishing points are $b$ and $g$. Hence, there must be at least six distinct poles, and the SD cannot be $5$-oriented.

Finally, assume that each arc contributes to at most two swirls, and recall that there are at least five swirls, each of which has degree at least three. Then, the arcs that contribute to such swirls are at least $\lceil (5\times 3) / 2 \rceil=8$. Consequently, a $5$-oriented SD has at least eight arcs.

As for SODs, recall that their swirl graphs are simple and bipartite, and each arc contributes to at most two swirls. Since each of the five swirls has degree at least three, and a bipartite graph on five vertices has at most six edges, we conclude that the arcs contributing to the five swirls are at least $5\times 3-6=9$. Thus, a $5$-oriented SOD has at least nine arcs.
\end{proof}

\subsection{6-Oriented SDs and Beyond}\label{sec:6.4}
As it turns out, $6$-oriented SDs and SODs with alignment $(3,3,3,3,3,3)$ already achieve the absolute minimum in terms of the number of swirls, as well as the number of arcs. 

\begin{theorem}\label{t:6orien}
Any SD has at least two clockwise swirls, two counterclockwise swirls, and six arcs. Any SOD has at least eight arcs. Moreover, there are matching examples that are $6$-oriented with alignment $(3,3,3,3,3,3)$.
\end{theorem}
\begin{proof}
We recalled in \cref{sec:previous} that any SOD has at least eight arcs; this was proved in~\cite{mini}. Also, we have proved in \cref{yt:22swirls} that any SD has at least two clockwise swirls and two counterclockwise swirls. It remains to prove that any SD has at least six arcs.

Let $S_1$ and $S_2$ be two clockwise swirls of an SD $\mathcal D$, of degrees $d_1\geq 3$ and $d_2\geq 3$, respectively. Without loss of generality, assume that $d_1\geq d_2$. Due to \cref{yp:adjswirl1}, $S_1$ and $S_2$ cannot be contiguous, because they are concordant.

Thus, by \cref{yc:noncontiguous}, $S_1$ and $S_2$ share at most $\lfloor d_2/2\rfloor$ arcs. As a consequence, the arcs involved in these two swirls are at most $d_1+d_2-\lfloor d_2/2\rfloor=d_1+\lceil d_2/2\rceil$. If $d_1\geq 4$, then $d_1+\lceil d_2/2\rceil\geq 6$, and $\mathcal D$ has at least six arcs.

So, assume that $d_1=d_2=3$ and exactly $d_1+\lceil d_2/2\rceil=5$ arcs are involved in $S_1$ and $S_2$. In this case, the two swirls share exactly one arc $a\in\mathcal D$. Let $\gamma$ be the great circle containing $a$. Of the arcs involved in $S_1$ and $S_2$, two hit $a$ and two block $a$; only the latter cross $\gamma$. Due to \cref{yp:3intercirc}, there must be a third arc that crosses $\gamma$, and so $\mathcal D$ has at least six arcs.

Matching examples with alignment $(3,3,3,3,3,3)$ are found in \cref{fig:333333}.
\end{proof}

\begin{figure*}
\centering
\begin{subfigure}{0.49\textwidth}
    \includegraphics[scale=0.5]{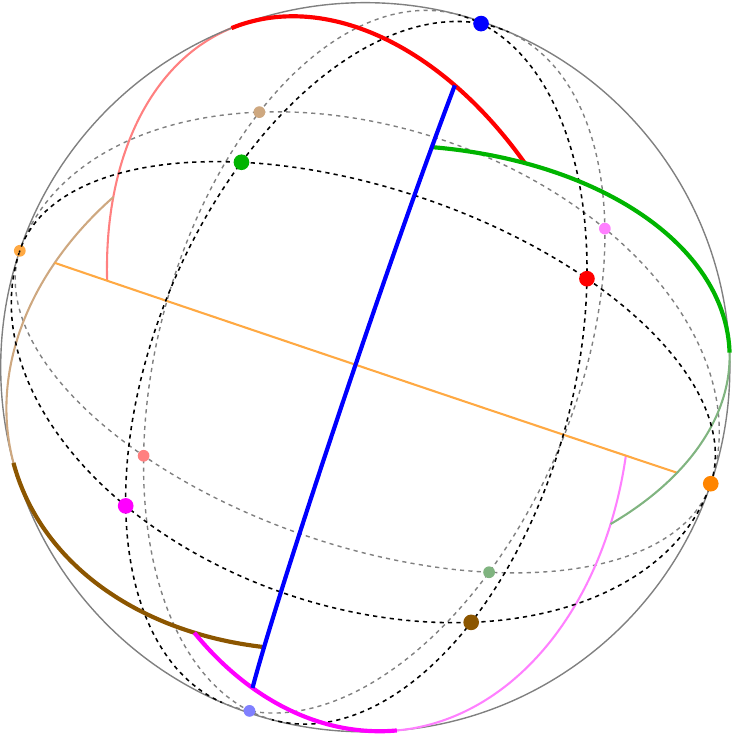}
    \subcaption{}
    \label{fig:333333SD}
\end{subfigure}
\hfill
\begin{subfigure}{0.49\textwidth}
    \includegraphics[scale=0.5]{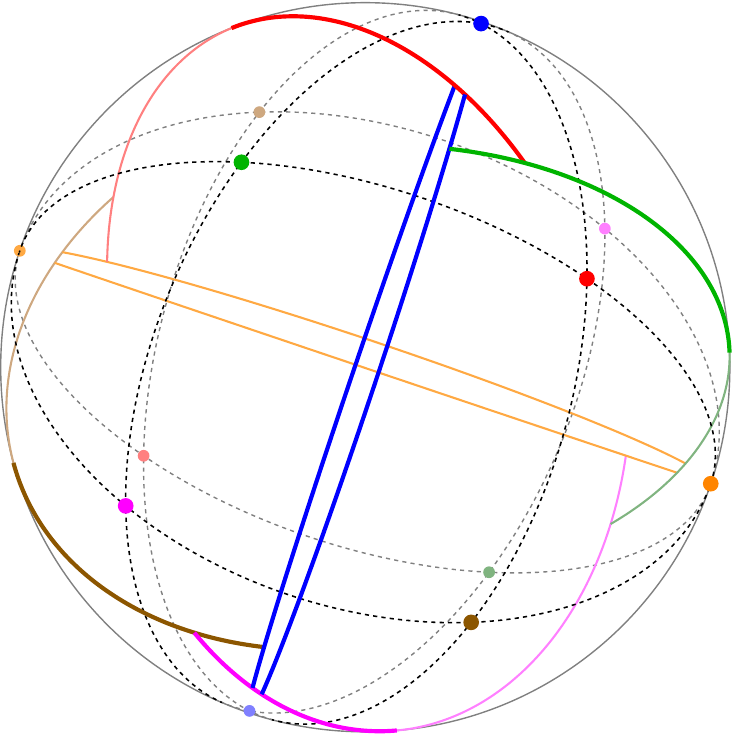}
    \subcaption{}
    \label{fig:333333SOD}
\end{subfigure}
\caption{(a) A $6$-oriented SD with alignment $(3,3,3,3,3,3)$ having exactly four swirls and six arcs. (b) A $6$-oriented SOD with alignment $(3,3,3,3,3,3)$ having exactly four swirls and eight arcs.}
\label{fig:333333}
\end{figure*}

By perturbing the poles of the SDs in \cref{fig:333333}, one obtains a continuum of configurations of $6$-oriented SDs and SODs having the absolute minimum number of swirls and arcs. These include degenerate and non-degenerate configurations of all the alignments listed in \cref{tab:1} for $k=6$ (except for $(2,2,2,2,2,5)$, which is covered in \cref{t:222k}). Adding ``dummy poles'' to these configurations yields similar results for $k>6$. In particular, this covers all non-degenerate configurations for any $k\geq 6$, i.e., those with alignment $(k-1,k-1,\dots,k-1)$.



\bibliography{minspherical}

\newpage

\appendix

\section{Previous Results}\label{asec:previous}
\subsection{Properties of SDs}
We summarize what is currently known about SDs. Proofs of the following statements are  found in~\cite{spherical2}; although they were stated and proved only for SODs, the reader may verify that none of the proofs makes use of the one-sidedness of the arrangements.

\begin{proposition}\label{p:04}
Every arc in an SD hits exactly two distinct arcs, one at each endpoint.\qed
\end{proposition}

\begin{proposition}\label{p:05}
No two arcs in an SD intersect in more than one point.\qed
\end{proposition}

\begin{proposition}\label{p:semi}
Given an SD $\mathcal D$, the relative interior of any great semicircle on the unit sphere is crossed by at least one arc of $\mathcal D$.\footnote{In~\cite{spherical2}, the word ``intersect'' is used instead of ``cross''. However, the proof therein actually yields this slightly stronger result.}\qed
\end{proposition}

\begin{proposition}\label{p:06}
An SD partitions the unit sphere into spherically convex regions.\qed
\end{proposition}

\begin{definition}
A \emph{tile} is each of the (spherically convex) regions into which the unit sphere is partitioned by an SD.
\end{definition}

\begin{corollary}\label{p:tileanti}
In an SD, no tile (including its boundary) contains two antipodal points.\qed
\end{corollary}

\begin{proposition}\label{p:07a}
Removing any one arc from an SD and taking the union of the remaining arcs yields a connected subset of the unit sphere.\qed
\end{proposition}

\begin{corollary}\label{c:conn}
The union of all the arcs in an SD is a connected set.\qed
\end{corollary}

\begin{proposition}\label{p:08}
An SD with $n$ arcs partitions the unit sphere into $n+2$ tiles.\qed
\end{proposition}

\begin{proposition}\label{p:eyeconv1}
In an SD $\mathcal D$, let $\mathcal S$ be a swirl of degree $d$.
\begin{itemize}
\item The union of the $d$ arcs of $\mathcal S$ separates the unit sphere in two regions, exactly one of which is spherically convex; this region is a spherical $d$-gon called the \emph{eye} $E$ of $\mathcal S$.
\item The only points of intersection between pairs of arcs of $\mathcal S$ are the vertices of $E$.\qed
\end{itemize}
\end{proposition}

\begin{proposition}\label{p:eyebound}
In an SD, if the eyes of two distinct swirls have intersecting interiors, then their boundaries are disjoint.\qed
\end{proposition}

\begin{definition}
The \emph{swirl graph} of an SD $\mathcal D$ is the undirected multigraph on the set of swirls of $\mathcal D$ having an edge between two swirls for every arc in $\mathcal D$ shared by the two swirls.
\end{definition}

\begin{theorem}\label{t:graph1}
The swirl graph of any SD is planar; moreover, any SD has at least one clockwise swirl and at least one counterclockwise swirl.\qed
\end{theorem}

The following results on SDs were proved in~\cite{mini}. Again, they were stated only for SODs, but their proofs do not require the one-sidedness of the arrangements.

\begin{lemma}\label{l:hemilemma}
Given an SD $\mathcal D$, the relative interior of any hemisphere contains the eye of at least one swirl of $\mathcal D$.\qed
\end{lemma}

\begin{theorem}\label{thm:swirls2}
Every SD has at least four swirls.\qed
\end{theorem}

\subsection{Properties of SODs}
In addition to the previous properties, SODs also enjoy the following ones. All proofs are found in~\cite{spherical2}.

\begin{proposition}\label{p:tileedge}
In an SOD, any arc coincides with an edge of a tile.\qed
\end{proposition}

\begin{proposition}\label{p:eyeconv2}
In an SOD $\mathcal D$, let $\mathcal S$ be a swirl of degree $d$, and let $E$ be its eye. Then, the tiles of $\mathcal D$ adjacent to $E$ are exactly $d$; any two such tiles are either disjoint or intersect only along a single arc of $\mathcal S$.\qed
\end{proposition}

\begin{theorem}\label{t:graph2}
The swirl graph of any SOD is a simple planar bipartite graph with non-empty partite sets.\qed
\end{theorem}

Since there are at least four swirls in an SOD (\cref{thm:swirls2}) and the swirl graph of an SOD is simple, planar and bipartite (\cref{t:graph2}), it easily follows that any SOD has at least eight arcs. A proof is found in~\cite{mini}.

\begin{corollary}\label{cor:8arcs}
Every SOD has at least eight arcs.\qed
\end{corollary}

\section{Sliding Walks}\label{asec:walks}
We will prove \cref{ylemma:right}; our technique is inspired by the proof of \cref{l:hemilemma} (cf.~\cite[Lemma~9]{mini}).

\begin{lemma}\label{lemma:right}
The right-side region (resp., left-side region) of any sliding walk on an SD $\mathcal D$ contains the eye of a clockwise swirl (resp., counterclockwise swirl) of $\mathcal D$.
\end{lemma}
\begin{proof}
Let $w$ be any sliding walk and let $A$ be its right-side region. Consider a right-handed walk $w'$ starting from any point on the boundary of $A$, as shown in \cref{fig:1a}. Note that $w'$ never leaves $A$, because it is a right-handed walk: upon reaching the boundary of $A$, $w'$ follows it clockwise until it reaches one of its vertices. Then it turns right, either remaining on the boundary of $A$ or entering its interior. In particular, the right-side region $E$ of $w'$ is contained in $A$. By \cref{yobs:right}, $E$ is the eye of a clockwise swirl.

For a similar reason, the left-side region of $w$ contains the eye of a counterclockwise swirl.
\end{proof}

\section{Arc Doubling}\label{asec:doubling}
An \emph{overlap} in an SD is any endpoint shared by two arcs.\footnote{In~\cite{spherical2}, SODs with overlaps are called ``degenerate''. We adopted a different terminology in this paper to avoid confusion with degenerate $k$-oriented SDs.} Observe that an overlap must lie in the relative interior of some arc.

Let $\mathcal D$ be an SD, and let $a\in \mathcal D$. For a sufficiently small positive real number $\epsilon$, we define the \emph{$\epsilon$-doubling} of $a$. This operation consists of replacing $a$ with two disjoint arcs $a'$ and $a''$ defined as follows. Let $x$ and $y$ be the endpoints of $a$, located in the relative interiors of arcs $b\in\mathcal D$ and $c\in\mathcal D$, respectively. Then, $a'$ and $a''$ also hit $b$ at $x'$ and $x''$, respectively, such that the geodesic arc $x'x''$ has length $\epsilon$ and its midpoint is $x$. Similarly, $a'$ and $a''$ hit $c$ at points $\epsilon$ apart with midpoint $y$. Moreover, any arc of $\mathcal D$ that hit $a$ or shared and endpoint with $a$ is now slightly shortened to hit either $a'$ or $a''$.

\begin{observation}\label{obs:doubling}
$\epsilon$-doubling an arc of an SD preserves the number of its swirls.\qed
\end{observation}

\begin{proposition}\label{p:doubling}
Any SD, possibly with overlaps, can be converted into an SOD without overlaps by $\epsilon$-doubling of arcs.
\end{proposition}
\begin{proof}
Observe that $\epsilon$-doubling an arc eliminates any overlaps involving that arc without introducing new overlaps, provided that $\epsilon$ is sufficiently small. Furthermore, if an arc is hit from both sides, the $\epsilon$-doubling operation replaces it with two arcs, each of which is hit only from one side.
\end{proof}

\section{Swirl Adjacency}\label{asec:swirl}
These are the missing proofs from \cref{sec:swirl}.

\begin{proposition}\label{p:contiguous}
In any SD, pairs of contiguous swirls share exactly two arcs.
\end{proposition}
\begin{proof}
Let $S_1$ and $S_2$ be contiguous swirls sharing two arcs $a$ and $b$, with $a$ hitting $b$ at point $p$. The great circles containing $a$ and $b$ subdivide the unit sphere into four spherical lunes. Each of the eyes of $S_1$ and $S_2$, being spherically convex, lies in one such spherical lune; moreover, the two eyes lie in lunes that are adjacent along $a$. Thus, if there is a third arc $c$ shared by $S_1$ and $S_2$, then $a$ must hit $c$ at the endpoint opposite to $p$. However, any arc of a swirl hits exactly one other arc in the same swirl, meaning that $c$ cannot be shared by $S_1$ and $S_2$.
\end{proof}

\begin{corollary}\label{c:contiguous}
If two swirls $S_1$ and $S_2$ of an SD $\mathcal D$ are contiguous along an arc $a\in\mathcal D$, then $\epsilon$-doubling $a$ separates the eyes of $S_1$ and $S_2$ and results in a swirl graph where $S_1$ and $S_2$ are connected by a single edge.
\end{corollary}
\begin{proof}
By \cref{p:contiguous}, the swirls $S_1$ and $S_2$ share exactly two arcs $a$ and $b$. Thus, $\epsilon$-doubling $a$ yields two swirls that share only $b$.
\end{proof}

\begin{proposition}\label{p:noncontiguous}
In any SD, non-contiguous swirls may only share arcs that are not consecutive in either swirl.
\end{proposition}
\begin{proof}
Assume that two swirls $S_1$ and $S_2$ share two arcs $a$ and $b$ that are consecutive in $S_1$, with $a$ hitting $b$ at a vertex $p$ of the eye of $S_1$. By \cref{p:eyeconv1} applied to $S_2$, we have that $p$ must be a vertex of the eye of $S_2$, as well. It follows that $S_1$ and $S_2$ are contiguous.
\end{proof}

\begin{corollary}\label{c:noncontiguous}
In any SD, a swirl of degree $d$ may share at most $\lfloor d/2\rfloor$ arcs with the same non-contiguous swirl.
\end{corollary}
\begin{proof}
In a cycle of $d$ arcs, any subset of more than $\lfloor d/2\rfloor$ arcs contains a pair of consecutive arcs. If these arcs were shared with the same non-contiguous swirl, \cref{p:noncontiguous} would be contradicted.
\end{proof}

\begin{proposition}\label{p:adjswirl1}
In an SD, let $S_1$ and $S_2$ be two swirls that share more than one arc. The following statements are equivalent.
\begin{itemize}
\item $S_1$ and $S_2$ are not contiguous.
\item $S_1$ and $S_2$ are concordant.
\item The eyes of $S_1$ and $S_2$ have antipodal interior points.
\end{itemize}
\end{proposition}
\begin{proof}
Contiguous swirls are obviously discordant. Also, their eyes lie on the same side of one of the two arcs shared by the two swirls. Hence, both eyes lie in the same hemisphere determined by that arc, and therefore cannot have antipodal (interior) points.

We have proved that the first statement is implied by the second and by the third. We will now prove the converse. Let $S_1$ be a swirl with eye $E_1$, and let $E'_1$ be the spherical polygon antipodal to $E_1$. Let $a$ and $b$ be two arcs that are shared between $S_1$ and a non-contiguous swirl $S_2$; by \cref{p:noncontiguous}, $a$ and $b$ are disjoint.

The two great circles containing $a$ and $b$ subdivide the unit sphere into four spherical lunes, one of which, say $L_1$, contains $E_1$. Assume that the lune containing the eye $E_2$ of $S_2$ is adjacent to $L_1$. Without loss of generality, $E_1$ and $E_2$ lie on the same side of $a$ and on opposite sides of $b$. However, this implies that $S_1$ and $S_2$ are contiguous along $b$, which is a contradiction.

We conclude that $E_2$ lies in the lune $L_2$ opposite to $L_1$. Note that $a$ and $b$, departing from $E_1$, follow the edges of $L_1$ in opposite directions. Then, they continue along the edges of $L_2$ in the same directions, implying that $S_1$ and $S_2$ are concordant swirls.

Let $p_1\in L_1$ and $p_2\in L_2$ be the endpoints of $a$, and let $q_1\in L_1$ and $q_2\in L_2$ be the endpoints of $b$. Since $p_1$ and $q_1$ are vertices of $E_1$, the segment $p_1q_1$ is contained in $E_1$. Likewise, $p_2q_2$ is contained in $E_2$. Since $a$ and $b$ are shorter than great semicircles, the points $p_2'$ and $q_2'$, antipodal to $p_2$ and $q_2$ respectively, lie outside of $a$ and $b$. Nonetheless, $p_2',q_2'\in L_1$. It follows that $p_1q_1$ and $p_2'q_2'$ cross each other at a point $x$ internal to $E_1$. The point $x'$ antipodal to $x$ lies in $p_2q_2$, and is therefore internal to $E_2$.
\end{proof}

\section{Attractors}\label{asec:attractors}
These are the missing proofs from \cref{sec:attractors}.

\begin{observation}\label{o:attractor}
Given any \emph{attractor} $A$ of a $k$-oriented SD $\mathcal D$, each arc of $\mathcal D$ has a unique vanishing point in $A$, and is collinear with it.\qed
\end{observation}

\begin{proposition}\label{p:noncoll}
The $k$ poles of a $k$-oriented SD cannot be all collinear. Thus, all attractor hulls have non-empty interiors.
\end{proposition}
\begin{proof}
Assume for a contradiction that all poles of a $k$-oriented SD $\mathcal D$ (and therefore all its anti-poles) lie on the same great circle $\gamma$. Due to \cref{p:semi}, there is an arc $a\in\mathcal D$ that crosses $\gamma$, say at $x$. By definition of $k$-oriented SD, $a$ is collinear with a pole $f(a)\neq x$. However, since $f(a)$ and $x$ lie on $\gamma$, then so does $a$, contradicting the fact that $a$ crosses $\gamma$ at $x$.

Thus, no $k$ points among poles and anti-poles can be collinear, and therefore their spherical convex hull must have a non-empty interior. It follows that all attractors hulls have non-empty interiors.
\end{proof}

\begin{proposition}\label{p:2cross}
If an arc of a $k$-oriented SD intersects the interior of an attractor hull $H$, it intersects the boundary of $H$ in at most one point.
\end{proposition}
\begin{proof}
By \cref{o:attractor}, any arc $a$ is collinear with a vanishing point $p$ within $H$. Also, by definition of $k$-oriented SD, $a$ does not contain $p$. It follows that, if $a$ intersects the interior of $H$, it cannot intersect its boundary in two points, because such points would be on opposite sides of $p$ (or coincident with $p$).
\end{proof}

\begin{proposition}\label{p:1inters}
If an arc of a $k$-oriented SD $\mathcal D$ intersects the boundary of an attractor hull $H$ at a point $x$, then there is an arc of $\mathcal D$ that crosses the boundary of $H$ at a point $y$, such that $x$ and $y$ are $H$-connected (with respect to $\mathcal D$) and lie on a same edge of $H$. Moreover, if $x\neq y$, then $y$ is not a vertex of $H$.
\end{proposition}
\begin{proof}
If $\mathcal D$ has an arc that crosses $H$ at $x$, there is nothing to prove. Otherwise, \cref{o:attractor} implies that $\mathcal D$ has an arc $a$ collinear with an edge $e$ of $H$ containing $x$. Moreover, $x$ lies in the interior of $a$; thus, $a$ partially overlaps with $e$. Hence, there is a vanishing point $v\in e$ that $a$ cannot reach, by \cref{o:attractor}. So, $a$ hits an arc $a'\in\mathcal D$ at a point $y\in e$, strictly between $x\in e$ and $v\in e$. Therefore, $y$ is internal to $e$, it is $H$-connected with $x$, and $a'$ crosses $e$ at $y$.
\end{proof}

\begin{proposition}\label{p:attractor}
Let $H$ be an attractor hull of a $k$-oriented SD $\mathcal D$. Then, there exist arcs of $\mathcal D$ that intersect the interior of $H$. Moreover, any point of intersection between an arc of $\mathcal D$ and the interior of $H$ is internally $H$-connected with the vertices of the eye of a swirl of $\mathcal D$.
\end{proposition}
\begin{proof}
Let $A$ be an attractor of $\mathcal D$, and let $H$ be the spherical convex hull of $A$. If $H$ is total, there is nothing to prove, because $\mathcal D$ is connected and it has at least a swirl (cf.~\cref{c:conn,yt:22swirls}).

Otherwise, assume for a contradiction that no arc of $\mathcal D$ intersects $H$. Then, $H$ is contained in the interior of a spherically convex tile $T$ (cf.~\cref{p:06}). If $a$ is any arc of $\mathcal D$ bounding $T$, then $T$ lies on one side of the great circle $\gamma$ containing $a$, and therefore $\gamma$ does not intersect $H$. Hence, $a$ is not collinear with any point in $A\subset H$, contradicting \cref{o:attractor}. Consequently, there is an arc of $\mathcal D$ that intersects $H$, and by \cref{p:1inters} there is also an arc that intersects the interior of $H$ (note that $H$ has an interior, due to \cref{p:noncoll}).

Now, let $x$ be any point of intersection between an arc of $\mathcal D$ and the interior of $H$, and let $w$ be a sliding walk that starts from $x$ and follows each arc in the direction of its vanishing point in $A$, which exists due to \cref{o:attractor}. An example of such a sliding walk is shown in \cref{fig:walk.a}. Since $H$ is convex and $w$ always moves toward a point of $A\subset H$ without ever reaching it, we conclude that $w$ never leaves the interior of $H$. Thus, either the left-side or the right-side region of $w$ is entirely contained in the interior of $H$. In turn, this region contains the eye of a swirl, due to \cref{lemma:right}.
\end{proof}

\begin{corollary}\label{c:attractoreye}
In any $k$-oriented SD $\mathcal D$, the interior of any attractor hull contains the eye of a swirl of $\mathcal D$.
\end{corollary}
\begin{proof}
If $H$ is an attractor hull, by \cref{p:attractor} there is a point $x$ of intersection between an arc of $\mathcal D$ and the interior of $H$. Also, there is an eye $E$ of a swirl of $\mathcal D$ whose vertices lie in the interior of $H$ and are $H$-connected with $x$. Since $H$ is spherically convex, its interior contains $E$.
\end{proof}

It is easy to see that, for a fixed $k$-oriented SD, any hemisphere contains an attractor hull. Such an attractor hull is unique, provided that no poles lie on the hemisphere's boundary. Thus, we may view \cref{c:attractoreye} as a stronger version of \cref{l:hemilemma} for $k$-oriented SDs.

\begin{proposition}\label{p:2linesides}
Let $R$ be a spherical polygon contained in the interior of a hemisphere of the unit sphere, and let $x$ be a point of an SD $\mathcal D$ that lies in the interior of $R$. Then, for every great circle $\gamma$ through $x$, there are arcs of $\mathcal D$ that thrust the boundary of $R$ at two distinct points $y$ and $z$ located on opposite sides of $\gamma$ (or on $\gamma$). Moreover, $y$ and $z$, are internally $R$-connected with $x$.
\end{proposition}
\begin{proof}
Let $p$ be a point of intersection between $\gamma$ and the boundary of a hemisphere containing $R$. Let $y$ (resp., $z$) be the first point where a clockwise (resp., counterclockwise) walk with fulcrum $p$ starting from $x$ intersects the boundary of $R$. In the special case where $x$ lies on an arc of $\mathcal D$ contained in $\gamma$, we choose the two walks to go in different directions, so $y$ and $z$ are necessarily distinct. Also, $y$ and $z$ clearly satisfy the desired conditions.
\end{proof}

\begin{proposition}\label{p:3inter}
Let $R$ be a spherical polygon contained in the interior of a hemisphere of the unit sphere, and let $\mathcal D$ be an SD having an arc with an endpoint $x$ internal to $R$. Then, there are arcs of $\mathcal D$ that thrust the boundary of $R$ in at least three distinct points that are all internally $R$-connected with $x$.
\end{proposition}
\begin{proof}
Let $x$ be an internal point of $R$ where an arc of $\mathcal D$ hits another arc. Let $\gamma$ be any great circle through $x$; by \cref{p:2linesides}, there are two distinct points $y$ and $z$ where arcs of $\mathcal D$ thrust the boundary of $R$. Moreover, $x$, $y$, and $z$ are $R$-connected.

Assume that $y$ and $z$ are not collinear with $x$. Then there is a great circle $\gamma'$ through $x$ such that $y$ and $z$ are strictly on the same side of $\gamma'$. By \cref{p:2linesides}, there is a third point $w$ lying on the opposite side of $\gamma'$ (or on $\gamma'$) where an arc of $\mathcal D$ thrusts the boundary of $R$. Moreover, $x$ and $w$ are internally $R$-connected, implying that $y$, $z$, and $w$ are internally $R$-connected.

Assume that $y$ and $z$ are collinear with $x$. Because $x$ is shared by two arcs $a,b\in\mathcal D$, there is a point $x'$ in $R\cap (a\cup b)$ such that $x'$ is not collinear with $y$ and $z$. Hence, there is a great circle $\gamma'$ through $x'$ such that $y$ and $z$ are strictly on the same side of $\gamma'$ (because $y$ and $z$ are not antipodal, since $R$ lies in the interior of a hemisphere). Now we can conclude the proof as in the previous case, observing that $x$ and $x'$ are internally $R$-connected.
\end{proof}

Observe that \cref{p:2linesides,p:3inter} do not require $R$ to be spherically convex or even simply connected.

We can also prove a version of \cref{p:3inter} in the limit case where $R$ is an entire hemisphere.

\begin{proposition}\label{p:3intercirc}
Given an SD $\mathcal D$, any great circle on the unit sphere is crossed by at least three arcs of $\mathcal D$.
\end{proposition}
\begin{proof}
Let $\gamma$ be a great circle, and let $p,p'\in\gamma$ be two antipodal points which divide $\gamma$ into two great semicircles. By \cref{p:semi}, the interior of each of these two semicircles is crossed by an arc of $\mathcal D$. Note that the two crossing points $x$ and $y$ are distinct, because the two semicircles have disjoint interiors. Now, $x$ and $y$ divide $\gamma$ into two arcs, at least one of which is not shorter than a great semicircle. The interior of this arc is crossed by an arc of $\mathcal D$, again by \cref{p:semi}. Thus, we have found three distinct points where arcs of $\mathcal D$ cross $\gamma$. Since a geodesic arc cannot cross a great circle in more than one point, there must be three distinct arcs of $\mathcal D$ crossing $\gamma$.
\end{proof}

We can improve \cref{p:2linesides,p:3inter} in the case where $R$ is an attractor hull.

\begin{proposition}\label{p:3linesides}
Let $H$ be a non-total attractor hull of a $k$-oriented SD $\mathcal D$, and let $x$ be a point of $\mathcal D$ that lies in the interior of $H$. Then, for every great circle $\gamma$ through $x$, there are three distinct arcs of $\mathcal D$ that thrust the boundary of $H$ at three distinct points, all internally $H$-connected with $x$, two of which lie on strictly opposite sides of $\gamma$ (i.e., not on $\gamma$).
\end{proposition}
\begin{proof}
Since any non-total attractor hull is a spherical polygon contained in the interior of a hemisphere, \cref{p:3inter} implies that there are at least three points, all internally $H$-connected with $x$, where arcs of $\mathcal D$ thrust the boundary of $H$. Moreover, each of these points belongs to a distinct arc of $\mathcal D$, due to \cref{p:2cross}. To conclude the proof, it suffices to show that two such points lie on opposite sides of $\gamma$.

\cref{p:04} implies that $x$ is internal to a unique arc $a\in\mathcal D$. By \cref{p:2cross}, at least one endpoint $x'$ of $a$ lies in the interior of $H$; $a$ hits another arc $a'\in\mathcal D$ at $x'$.

Assume first that $a$ crosses $\gamma$ at $x$, i.e., $a$ does not lie on $\gamma$. Proceeding as in the proof of \cref{p:2linesides}, we construct two paths from $x$, following $a$ in opposite directions away from $\gamma$. The points $y$ and $z$ thus obtained are therefore on strictly opposite sides of $\gamma$.

Assume now that $a$ does not cross $\gamma$, hence it lies on it. Therefore, $a'$ crosses $\gamma$ at $x'$, which is internally $H$-connected with $x$. Thus, we can repeat the previous argument with $x'$ and $a'$ in lieu of $x$ and $a$.
\end{proof}

\begin{proposition}\label{l:3interhull}
Let $H$ be a non-total attractor hull of a $k$-oriented SD $\mathcal D$, and let $x$ be a point of intersection between an arc of $\mathcal D$ and the interior of $H$. Then, there are three distinct arcs of $\mathcal D$ that thrust the boundary of $H$ at three distinct points, not all lying on the same edge of $H$, all of which are internally $H$-connected with $x$.
\end{proposition}
\begin{proof}
By \cref{p:3linesides}, there are three arcs of $\mathcal D$ that thrust the boundary of $H$ at distinct points, all internally $H$-connected with $x$. If no edge of $H$ contains all such points, we are finished. Otherwise, all such points lie on the same edge $e$ of $H$. Also, $e$ is unique, because at least one of the three points is internal to $e$. Let $\gamma$ be a great circle through $x$ that does not intersect $e$ (note that $\gamma$ exists because $e$ is shorter than a great semicircle). By \cref{p:3linesides} there is a point, internally $H$-connected with $x$ and not lying on $e$, where an arc of $\mathcal D$ thrusts the boundary of $H$. Such an arc is distinct from the previous ones due to \cref{p:2cross}.
\end{proof}

It is easy to see that the previous results cannot be improved, as there are $k$-oriented SDs whose arcs thrust the boundary of an attractor hull at exactly three points, two of which lie on the same edge of the attractor hull.

Let $A$ be an attractor of and SD $\mathcal D$, and let $H$ be the attractor hull relative to $A$, i.e., the spherical convex hull of $A$. A point of $A$ is a \emph{boundary point} if it lies on the boundary of $H$, and is an \emph{internal point} if it lies in the interior of $H$.

\begin{proposition}\label{p:hulltri}
Let $H$ be a non-total attractor hull relative to an attractor $A$ of an SD $\mathcal D$, and let $h_1$, $h_2$, \dots, $h_m$ be the boundary points of $A$, taken in clockwise order. Let $3\leq i\leq m$, and let $p$ be the point of intersection between $h_1h_{i-1}$ and $h_2h_i$. Assume that the interiors of the triangles $h_1h_2p$ and $h_{i-1}h_ip$ are devoid of points of $A$, while the triangle $h_1ph_i$ contains at most one internal point of $A$. Then, if a point $x$ of $\mathcal D$ lies in the interior of $h_1ph_i$, there exists an arc of $\mathcal D$ that crosses the interior of $h_1h_i$ at a point internally $H$-connected with $x$.
\end{proposition}
\begin{proof}
Let $q$ be the unique internal point of $A$ contained in $h_1ph_i$, if such point exists. Let $w$ be a sliding walk starting from $x$ with the following properties. As long as $w$ is in the interior of an arc $a\in\mathcal D$ with a vanishing point $v\in h_1h_ih_{i+1}\dots h_m$, $w$ proceeds along $a$ in the direction of $v$. Instead, if $a$ has a vanishing point $v\in ph_2h_3\dots h_{i-1}$, then $w$ follows $a$ in the direction away from $v$. If $q$ is a vanishing point of $a$, then $w$ follows $a$ in the direction of $q$ if and only if $w$ is outside of the triangle $h_1qh_i$.

Observe that $h_1ph_i$ is the intersection of all the triangles of the form $h_1vh_i$, with $v$ in the spherical polygon $ph_2h_3\dots h_{i-1}$. Therefore, $w$ always remains within the interior of $h_1ph_i$, approaching the triangle $h_1qh_i$ as long as it is outside of it, and then approaching $h_1h_i$ while it is in $h_1qh_i$. Thus, $w$ eventually crosses the interior of $h_1h_i$ at a point internally $H$-connected with $x$.
\end{proof}

Note that \cref{p:hulltri} cannot be improved, as there are counterexamples where $h_1h_2p$ or $h_{i-1}h_ip$ contains an internal attractor point or $h_1ph_i$ contains two internal attractor points.

\begin{corollary}\label{c:hulltri}
With the notation of \cref{p:hulltri}, if a point $x$ of $\mathcal D$ lies in the interior of $h_1h_2h_3$, and $h_1h_2h_3$ contains at most one internal point of $A$, then there is an arc of $\mathcal D$ that crosses the interior of $h_1h_3$ at a point internally $H$-connected with $x$.
\end{corollary}
\begin{proof}
It is immediate from \cref{p:hulltri} with $i=3$. Note that $p=h_2$, and therefore the triangles $h_1h_2p$ and $h_{i-1}h_ip$ degenerate to the edges $h_1h_2$ and $h_2h_3$. Since the interiors of such edges are empty, they are automatically devoid of points of $A$.
\end{proof}

We say that an attractor hull of a $k$-oriented SD is \emph{void} if it is a spherical $k$-gon, i.e., there are no poles or anti-poles in its interior. In the case of void attractor hulls, we have an improved version of  \cref{l:3interhull}.

\begin{proposition}\label{l:3edgehull}
Let $H$ be a void non-total attractor hull of a $k$-oriented SD $\mathcal D$, and let $x$ be a point of $\mathcal D$ that lies in the interior of $H$. Then, there are at least three distinct arcs of $\mathcal D$ that cross (respectively, thrust) the boundary of $H$ at distinct points, not all lying on the same two edges of $H$, all of which are $H$-connected (respectively, internally $H$-connected) with $x$.
\end{proposition}
\begin{proof}
Note that it suffices to prove that there are arcs of $\mathcal D$ that thrust (rather than cross) the boundary of $H$ in at least three distinct points, not all lying on the same two edges of $H$, all of which are internally $H$-connected with $x$. Indeed, given these points, \cref{p:1inters} easily implies that there are also at least three crossing points with the desired properties. In addition, \cref{p:2cross} implies that all such points lie on distinct arcs of $\mathcal D$.

Let $Y$ be the set of points on the boundary of $H$ such that, for all $y\in Y$ there is an arc of $\mathcal D$ that thrusts the boundary of $H$ at $y$, and $y$ is internally $H$-connected with $x$. By \cref{l:3interhull}, $Y$ contains at least three points, and no single edge of $H$ contains all of them. Assume for a contradiction that all points of $Y$ lie within two edges of $H$, say $e_1=h_1h_2$ and $e_2=h_{i-1}h_{i}$. Without loss of generality, $e_1$ contains at least two points of $Y$, say $y_1$ and $y_2$, and $e_2$ contains at least one point of $Y$, say $y_3$.

Recall that $y_1$ and $y_3$ are internally $H$-connected (with respect to $\mathcal D$) by a path $P_1$ along $\mathcal D$. Similarly, $y_2$ and $y_3$ are internally $H$-connected (with respect to $\mathcal D$) by a path $P_2$. It is easy to recognize that the union of $P_1$ or $P_2$ (whose points are all internal to $H$, except for $y_1$, $y_2$, $y_3$) contains a point $x'$ such that either $x'=p$ or $x'$ lies in the interior of one of the triangles $h_1ph_i$ and $h_2h_{i-1}p$.

We can exclude the case $x'=p$, because any arc $a\in \mathcal D$ that has $p$ in its interior and avoids the interiors of $h_1ph_i$ and $h_2h_{i-1}p$ must overlap with $h_1h_{i-1}$ or with $h_2h_i$. This is because $a$ is collinear with a vertex of $H$, and there are no vertices of $H$ in the interior of $h_1h_2$ or in the interior of $h_{i-1}h_i$. Then, by \cref{p:2cross}, $a$ has an endpoint in the interior of $h_1h_{i-1}$ or in the interior of $h_2h_i$, where it hits another arc $a'\in \mathcal D$, which in turn intersects the interior of $h_1ph_i$ or $h_2h_{i-1}p$, respectively.

Therefore, without loss of generality, we may assume that $x'$ is in the interior of $h_1ph_i$. Also, $x'$ is internally $H$-connected with $x$, because it lies on $P_1$ or on $P_2$. Note that $H$ and $x'$ satisfy the hypotheses of \cref{p:hulltri}, because $H$ is void. Thus, there is an arc of $\mathcal D$ that crosses $h_1h_i$ at an internal point $x''$, which is internally $H$-connected with $x'$ (hence with $x$). If $h_1h_i$ is an edge of $H$, we have reached a contradiction, because $x''\in Y$, but $x''$ is not in $e_1$ or in $e_2$.

Otherwise, $h_1h_i$ is a diagonal of $H$; let $\gamma$ be the great circle containing $h_1h_i$. By \cref{p:3linesides}, since $x''$ lies in the interior of $H$, there is an arc of $\mathcal D$ that thrusts the boundary of $H$ at a point $z$ on the side of $\gamma$ opposite to $e_1$ and $e_2$. Moreover, $z$ is internally $H$-connected with $x''$, and therefore with $x$. Hence $z\in Y$, which contradicts the fact that $z$ is not in $e_1$ or in $e_2$.

We conclude that no two edges of $H$ contain all points of $Y$, as desired.
\end{proof}

Again, \cref{l:3edgehull} cannot be improved, in the sense that there are counterexamples where exactly one attractor point lies in the interior of $H$.

\section{Open Problems}\label{asec:open}
We conclude this paper with some directions for future research.

Although we have proved all the results listed in \cref{tab:1}, which include all non-degenerate configurations of $k$-oriented SDs and SODs, several degenerate configurations are still unexplored. These include $6$-oriented SDs with alignment $(2,3,3,3,4,4)$ and $(4,4,4,4,4,4)$, both of which have two triplets of collinear poles (in the former case, one pole lies at the intersection of the two great circles containing all other poles). In fact, even classifying all degenerate configurations of $k\geq 6$ poles is a challenge in itself.

\begin{open}
Extend \cref{tab:1} to all degenerate $k$-oriented SDs and SODs for $k\geq 6$.
\end{open}

In particular, \cref{tab:1} still leaves a small uncertainty on the minimum number of swirls in a $5$-oriented SD with alignment $(3,3,3,4,4)$, which is either four or five. Although the analysis in \cref{l:5nondeg} can give insights on this configuration, it is not quite sufficient.

\begin{open}
Determine if there are $5$-oriented SDs with alignment $(3,3,3,4,4)$ having exactly four swirls.
\end{open}

\cref{fig:44444} shows a non-degenerate $5$-oriented SD with exactly four swirls and nine arcs. By $\epsilon$-doubling the green and the red arc, we obtain a non-degenerate $5$-oriented SOD with exactly four swirls and $11$ arcs. However, there is still a gap between this number and the absolute minimum number of arcs in a non-degenerate $5$-oriented SOD, which is nine.

\begin{open}
Determine if there exist non-degenerate $5$-oriented SODs with exactly four swirls and fewer than $11$ arcs.
\end{open}

In \cref{sec:1.1}, we remarked that, for any sufficiently large $k$, there are $k$-oriented SODs that are not the visibility map of any vertex-hidden point in a polygonal scene, due to the main result of~\cite{kimberly}. It would be interesting to determine if this is true for all $k$.

\begin{open}
Determine all values of $k$ such that any $k$-oriented SOD is the visibility map of a vertex-hidden point in a $k$-edge-oriented polygonal scene.
\end{open}

Furthermore, to address our original motivating question in \cref{sec:1.2}, we would like to find examples of $k$-edge-oriented polygonal scenes matching the numbers in the last column of \cref{tab:1}. The main contribution of~\cite{mini} is a first step in this direction, as it provides a polygonal scene where a vertex-hidden point sees exactly eight edges.

\begin{open}
For every $k$, determine whether there are $k$-edge-oriented polygonal scenes where a vertex-hidden point sees a number of edges that matches the minimum number of arcs in a non-degenerate $k$-oriented SOD.
\end{open}

In \cref{yt:22swirls}, we proved that any SD has at least two clockwise swirls and two counterclockwise swirls. However, this says nothing about whether these swirls are contiguous.

\begin{open}
Determine if all SDs have at least two clockwise swirls and two counterclockwise swirls, all of which are non-contiguous.
\end{open}

\cref{yt:22swirls} cannot be improved for $4$-oriented SDs, since it is possible to slightly modify the SD in \cref{fig:2223SD} to obtain a $4$-oriented SD with alignment $(2,2,2,3)$ having exactly two clockwise swirls. Doubling some of its edges also yields an SOD with the same property. However, perhaps \cref{yt:22swirls} can be improved for $3$-oriented SDs.

\begin{open}
Determine if all $3$-oriented SDs have at least four clockwise swirls and four counterclockwise swirls.
\end{open}

Since every octant of a $3$-oriented SD is a void attractor hull, it immediately follows from \cref{c:hulltri} that the interiors of all three edges of each octant must be crossed by arcs. Thus, the swirl graph of any $3$-oriented SD with exactly $12$ arcs has an induced subgraph isomorphic to the cubical graph. However, this is not enough to characterize the swirl graphs of such minimal $3$-oriented SDs: in fact, there are $3$-oriented SDs that have more than eight swirls in spite of having exactly $12$ arcs (some swirls may revolve around poles and be contiguous to some of the other swirls).

\begin{open}
Characterize the swirl graphs of $3$-oriented SDs with exactly $12$ arcs.
\end{open}

Observe that, in a $k$-oriented SD, the degree of a swirl cannot exceed $2k$, because at most two arcs in the swirl may have the same vanishing points. Nonetheless, the degree of a swirl may very well exceed $k$, although it is shown in~\cite{spherical2} that any \emph{swirling} SOD must have a swirl of degree exactly three. Still, not much is known about non-swirling SODs.

\begin{open}
Determine if there is a $k$-oriented SOD where all swirls have degrees exceeding $k$.
\end{open}

\cref{t:6orien} implies that there is a continuum of configurations of six poles that allow for the construction of $6$-oriented SDs with exactly six arcs. In principle, however, there may be non-degenerate configurations of six (or more) poles that make this construction impossible.

\begin{open}
Determine if it is always possible to construct a non-degenerate $6$-oriented SD with exactly six arcs, given the locations of its six poles.
\end{open}
\end{document}